\documentclass[10pt]{article}
\usepackage{color}
\usepackage{amssymb}
\usepackage{amsthm,array,amssymb,amscd,amsfonts,latexsym, url}
\usepackage{amsmath}
\usepackage[all]{xy}
\newtheorem{theo}{Theorem}[section]
\newtheorem{prop}[theo]{Proposition}
\newtheorem{claim}[theo]{Claim}
\newtheorem{lemm}[theo]{Lemma}
\newtheorem{coro}[theo]{Corollary}
\newtheorem{rema}[theo]{Remark}
\newtheorem{Defi}[theo]{Definition}
\newtheorem{ex}[theo]{Example}

\newtheorem{conj}[theo]{Conjecture}

\newtheorem{question}[theo]{Question}

\title{On Chern classes of   Lagrangian fibered hyper-K\"ahler manifolds}
\author{Claire Voisin\footnote{The author is supported by the ERC Synergy Grant HyperK (Grant agreement No. 854361).}}

\date{}

\newfont{\gothic}{eufb10}
\begin{document}
\maketitle
\begin{center} {\it \`{A} la m\'{e}moire de Jean-Pierre Demailly}
\end{center}

\begin{abstract} We study the rank stratification  for the differential of a Lagrangian fibration over a smooth basis. We also introduce and study the notion of Lagrangian morphism of vector bundles. As a consequence, we prove some of  the vanishing, in the Chow groups of a Lagrangian fibered hyper-K\"{a}hler variety $X$, of certain polynomials in the Chern classes of $X$ and the Lagrangian divisor,   predicted by the Beauville-Voisin conjecture. Under some natural assumptions on the dimensions  of the rank strata, we also establish  nonnegativity and positivity  results for  Chern classes.
 \end{abstract}
\section{Introduction}
In the paper \cite{beauvoi}, Beauville and myself proved that for any projective $K3$ surface $S$, there exists a canonical degree $1$ zero-cycle $o_S\in {\rm CH}_0(S)$ satisfying the following properties:

\begin{enumerate}
\item  \label{i} For any elements $D,\,D'\in {\rm CH}^1(S)$, $D\cdot D'={\rm deg}\,(D\cdot D') o_S$ in ${\rm CH}_0(S)$.
\item \label{ii}  $c_2(S)=24 o_S$ in ${\rm CH}_0(S)$.
\end{enumerate}
Although we can work with integral coefficients in the statement above,    ${\rm CH}(X)$ will denote from now on the Chow groups of a variety $X$ with $\mathbb{Q}$-coefficients.
Beauville subsequently proposed in \cite{beauspli} a  generalization   of property \ref{i} to any projective hyper-K\"{a}hler manifold $X$, whose weak version is called ``Beauville weak splitting conjecture", stating that any cycle $Z$ on $X$  which is a polynomial in divisor classes and is cohomologous to $0$ on $X$ is rationally equivalent to $0$ modulo torsion.   I proposed in turn  in \cite{voisinpamq} to generalize the combination of \ref{i} and \ref{ii}   to any projective hyper-K\"{a}hler manifold $X$ in the following form:
\begin{conj} \label{conjbeauvoi} Let $X$ be a hyper-K\"{a}hler manifold and let $Z\in{\rm CH}(X)$ be a cycle on $X$ which belongs to the subalgebra generated by divisor classes and the Chern classes of $X$. Then if $Z$ is cohomologous to $0$, $Z$ is rationally equivalent to $0$ modulo torsion on $X$.
\end{conj}
This conjecture  is known to hold for  Hilbert schemes of $K3$ surfaces by \cite{mauliknegut} and for generalized Kummer varieties by \cite{liefu}. It is also proved    in  \cite{voisinpamq}  for the Fano variety of lines in a smooth  cubic fourfold.
 Riess proved in \cite{riess} a very nice result concerning the weak splitting conjecture. We know  the polynomial cohomological  relations between divisor classes on a hyper-K\"{a}hler manifold $X$, as they were described by Verbitsky (see \cite{bogomolov}). When there is at least one divisor class $h=c_1(H)\in {\rm CH}^1(X)$ such that $\int_Xh^{2n}=0$, or equivalently $q(h)=0$, where $q$ is the Beauville-Bogomolov quadratic form,  the ideal of these  relations is generated by
  \begin{eqnarray}\label{relverb} l^{n+1}=0\,\,{\rm in}\,\,H^{2n+2}(X,\mathbb{Q}) \,\,{\rm if}\,\,q(l)=0.
  \end{eqnarray}
  (In the absence of a rational  isotropic class, the relations do not admit such a concrete description and only have a representation-theoretic characterization, unless we pass to real coefficients.)  Riess noticed that the relation (\ref{relverb}) obviously holds in ${\rm CH}^{n+1}(X)$ when $l=c_1(L)$ and $L$ is a Lagrangian line bundle, that is, there exist a Lagrangian fibration
  $f: X\rightarrow B$, an integer $d>0$,  and an ample line bundle $H_B$ on $B$, such that $L^{\otimes d}=f^* H_B$. Indeed, one has ${\rm dim}\,B=n$ so $H_B^{n+1}=0$ in ${\rm CH}^{n+1}(B)$ and
   \begin{eqnarray}\label{relverbchow}L^{n+1}=0\,\,\, {\rm in}\,\,\,{\rm CH}^{n+1}(X).
  \end{eqnarray}

   If an isotropic line bundle $L$ is  Lagrangian, $L$ is in particular  nef. The  SYZ conjecture states conversely that a nef isotropic line bundle on a hyper-K\"{a}hler manifold is Lagrangian.  When the Picard number of $X$  is $2$ and $X$ has an isotropic class, $X$ has two isotropic classes (up to a scalar), and  there are many instances when only one class is nef, the other ray of the positive cone not even belonging to the birational K\"{a}hler cone (see for example \cite{sacca}, where the case of the O'Grady 10-dimensional manifold constructed in \cite{lsv} is discussed; see also  \cite{DHMV}). We thus a priori do not have  the relation (\ref{relverbchow}) for the other  isotropic class. Riess, using work of Huybrechts \cite{huy}  on  the existence of  self-correspondences inducing  automorphisms of the ring ${\rm CH}(X)$,  could nevertheless  extend to the nonnef ray the relation (\ref{relverbchow}) and finally  prove the following:
  \begin{theo} (Riess \cite{riess})  If $X$ is a projective  hyper-K\"{a}hler manifold which has an isotropic class in ${\rm NS}(X)$ and  whose deformations  satisfy the  SYZ conjecture,  $X$ satisfies Beauville's weak splitting conjecture.
  \end{theo}

  One of  the  purposes in this paper is to study  Conjecture \ref{conjbeauvoi} for a certain type of  polynomial cohomological relations involving both  Chern classes of $X$ and divisor classes.
  More precisely, as proved in \cite{huy}, the following generalized Verbitsky relations hold in $H^*(X,\mathbb{Q})$ for $X$ hyper-K\"{a}hler of dimension $2n$
  \begin{eqnarray}\label{relverbchern} l^{n+1-i}c_I=0\,\,{\rm in}\,\,H^{2n+2-2i+4j}(X,\mathbb{Q}) \,\,{\rm if}\,\,q(l)=0 \,\,{\rm and}\,\,\,\,{\rm deg}\,c_I=4j\geq 4i,
  \end{eqnarray}
  where $c_I$ is any polynomial in the topological Chern classes $c_{2k}(X)\in H^{4k}(X,\mathbb{Q})$ of $X$ and the degree is the weighted cohomological degree. The relations (\ref{relverbchern}) hold in the cohomology algebra of $X$ because they hold for the class $\sigma_t$ of any $(2,0)$-form on a deformation $X_t$ of $X$, and these classes $\sigma_t$ fill-in an Euclidean  open (hence Zariski dense) set in the quadric $q=0$ in $H^2(X,\mathbb{C})$. Indeed $c_I$ is of Hodge type $(2j,2j)$ on $X_t$, while $\sigma^{n+1-i}$ is of Hodge type $(2(n+1-i),0)$ on $X_t$.
  \begin{rema}{\rm The odd Chern classes of a hyper-K\"{a}hler manifold $X$ vanish in ${\rm CH}(X)$ since its cotangent bundle is isomorphic to its dual. This is why we discuss only even Chern classes.}
  \end{rema}
\begin{rema}{\rm The relations (\ref{relverbchern}) do not  exhaust the cohomological relations in the tautological ring generated by divisor classes and topological Chern classes. For example, in top degree $4n$, the cohomology ring has rank $1$ and thus all polynomials of weighted degree $4n$ in the topological Chern classes $c_2(X),\ldots, c_{2n}(X)$  generate only a $1$-dimensional vector space. In degree $4n$, there are thus plenty of polynomial cohomological relations involving only the $c_{2i}(X)$, while the relations in (\ref{relverbchern}) not involving $l$ appear only for $i=n+1$, that is, in degree $4n+4$, (so they trivially hold in this case).}
\end{rema}

Conjecture \ref{conjbeauvoi} combined with (\ref{relverbchern})  leads us to the following
\begin{conj} \label{conjvan} Let $X$  be a    projective hyper-K\"{a}hler manifold. Then  for any line bundle $L$ on $X$ with $q(L)=0$, for   any integer $j\geq i$, and  any Chern monomial $c_I\in {\rm CH}^{2j}(X)$
 \begin{eqnarray}\label{relverbchowchernpredict} L^{n-i+1}c_I=0 \,\,\, {\rm in}\,\,\,{\rm CH}(X).
  \end{eqnarray}
  \end{conj}

  When   $L$ is   a  Lagrangian line bundle, the vanishing (\ref{relverbchowchernpredict}) is the vanishing (\ref{relverbchow}) for  $i=0$. The next case where $i=1$ is also quite easy. Indeed, we have to prove that
 \begin{eqnarray}\label{relverbchowchern1} L^nc_I=0 \,\,\, {\rm in}\,\,\,{\rm CH}^{n+2i}(X)
  \end{eqnarray}
  for any Chern polynomial of degree $2i>0$ on $X$. This follows from the fact that $L^n$ is proportional to the class of a general fiber $X_b$ of the Lagrangian fibration $f:X\rightarrow B$ and from the exact sequence
  \begin{eqnarray}\label{exseq} 0\rightarrow f_b^*\Omega_{B,b}\rightarrow \Omega_{X\mid X_b}\rightarrow \Omega_{X_b} \rightarrow 0,\end{eqnarray}
  where the vector bundles on the right and on the left are trivial on $X_b$ which is an abelian variety. This implies that $c_i(\Omega_{X})_{\mid X_b}=0$ in ${\rm CH}^i(X_b)$ for $i>0$.

  Note that we have more generally an exact sequence
  $$0\rightarrow f^*\Omega_{B^0}\rightarrow \Omega_{X^0}\rightarrow f^*T_{B^0}\rightarrow 0$$
  on $X^0:=f^{-1}(B^0)$, where $B^0\subset B$ is the Zariski open set where $B$ is smooth and over which $f$ has maximal rank.
  It follows that the Chern classes  of positive degree of $X$ are supported over a proper closed algebraic subset of $B$. A  strengthened conjecture in the case of a Lagrangian fibration is the following.
   \begin{conj}\label{conjvanop} Assume $X$ is projective  hyper-K\"{a}hler and  $f:X\rightarrow B$ is a  Lagrangian fibration. Then for any $ i$ and any  monomial $c_I\in {\rm CH}^{2i}(X)$ of weighted degree $2i$  in the Chern classes of $X$, there exists a codimension $i$ closed algebraic subset $B^i\subset B$, such that $c_I$ vanishes in ${\rm CH}(X\setminus f^{-1}(B^i))$.\end{conj}
   This conjecture obviously implies Conjecture \ref{conjvan} in the case where $L$ is Lagrangian, since then  $L^{n-i+1}$ is supported on $f^{-1}(B_{i-1})$, for a closed algebraic subset $B_{i-1}$ of $B$ of dimension $i-1$ and in general position.

  Our first results  in this paper concern   the study of the relations (\ref{relverbchowchernpredict}.  We first prove
   \begin{theo} \label{theonew}  Let $ f:X\rightarrow B$ be a Lagrangian fibration of a projective hyper-K\"{a}hler manifold, with Lagrangian line bundle $L$.
Then for  any pair of integers $i,\,j\leq n,\,j\geq i$, we have
   \begin{eqnarray} \label{relchernclassverbchow} L^{n-i+1} c_{2j}(X)=0\,\,{\rm in}\,\,{\rm CH}(X).\end{eqnarray}
  \end{theo}
  This proves   Conjecture \ref{conjvan}  only for the Chern classes, not for monomials in the Chern classes.
   We will establish Theorem \ref{theonew} in Section \ref{sectheonew}. We  will discuss   the case of $c_2^2L^{n-1}$ in Section \ref{secc2square}. As we will see, unlike (\ref{relchernclassverbchow}),  the vanishing
     \begin{eqnarray} \label{relchernclassverbchowavecsquare}L^{n-1}c_2^2(X)=0\,\,{\rm in}\,\,{\rm CH}(X)\end{eqnarray}
   predicted by Conjecture \ref{conjvan} does not follow from formal arguments involving the Lagrangian fibration. We will establish  this vanishing in Section \ref{secc2square} for a  specific class of Lagrangian fibered hyper-K\"{a}hler manifolds, namely the deformations of the $10$-dimensional  O'Grady manifolds constructed in \cite{lsv} as compactifications of the intermediate Jacobian fibration associated with a cubic fourfold (which we will call below a LSV-manifold).
   \begin{theo} \label{theoLSV} Let $X\rightarrow \mathbb{P}^5$ be a LSV-manifold. Then $c_2^2L^4=0$ in ${\rm CH}^8(X)$.
   \end{theo}

 In Section \ref{secriess}, we will show how the arguments of Riess in \cite{riess}  allow us to deduce from Theorem \ref{theonew} the following result in direction of Conjecture \ref{conjvan}.
\begin{theo}\label{theoriess}  Let $X$ be a projective hyper-K\"{a}hler manifold of dimension $2n$.   Assume the hyper-K\"{a}hler manifolds in the same deformation class as $X$ satisfy the SYZ conjecture. Then for any isotropic class $l=c_1(L)\in{\rm NS}(X)={\rm CH}^1(X)$, and for any pair of integers $i,\,j\leq n,\,j\geq i$,   one has
$$ L^{n+1-i} c_{2j}(X)=0\,\,{\rm  in }\,\,{\rm CH}(X).$$
\end{theo}
Section \ref{seclagmor} of the paper is devoted to introducing and studying the notion of {\it Lagrangian morphism}  between bundles $F,\,E$, of respective ranks $n$ and $2n$ on a variety $Y$, where $E$ is equipped with an everywhere nondegenerate skew-symmetric $2$-form, with the following
\begin{Defi} \label{defimorlag} A morphism $\phi: F\rightarrow E$ of vector bundles as above over $Y$ is Lagrangian if
\begin{enumerate}
\item The generic rank of $\phi$ is equal to $n$.

\item  At any  point $x\in Y$, the image ${\rm Im}\,\phi_x\subset E_x$ is isotropic  for $\langle\,,\,\rangle$ (hence it is   Lagrangian  at a point where the rank of $\phi$ is $n$).
\end{enumerate}
\end{Defi}

 The morphism given by the differential of a Lagrangian fibration map (over the smooth locus of the base) is the motivating example, although it is not general since    its rank loci are not of the expected codimension, except for the first (see Remark \ref{remageoranksup}). Nevertheless,  the following result applies to them when there are no nonreduced fibers in codimension $1$, (and we will explain in Section \ref{secmultiple} a variant if nonreduced fibers appear in codimension $1$).
\begin{theo} \label{theonewlag} Let $\phi:F\rightarrow E$ be a Lagrangian morphism of vector bundles on $Y$. Assume that $\phi$ is injective in codimension $1$. Then

(i) The locus $Y_{n-1}$ where $\phi$ has rank $ n-1$  is of codimension $2$ (or is empty).

(ii) $Y_{n-1}$ is a local complete intersection.

(iii) There is an exact sequence
\begin{eqnarray} \label{eactintro} 0\rightarrow F\rightarrow E\rightarrow F^*\rightarrow \mathcal{G}\rightarrow 0,
\end{eqnarray}
where $\mathcal{G}$ is a torsion sheaf supported on the locus $Y_{\leq n-1}$ where $\phi$ has rank $\leq n-1$, and $\mathcal{G}$ is a line bundle on $Y_{n-1}$.
\end{theo}
The geometry of this sheaf $\mathcal{G}$ seems very interesting, as it has a locally free resolution of length $3$.

 In Section \ref{secranklocilagfib}, we will study the loci $X_{k}$, resp. $X_{\leq k}$, where a Lagrangian fibration morphism has rank $k$, resp. $\leq k$. We will  explain why, starting from corank $2$,   they do  not satisfy the general codimension estimates for  general Lagrangian morphisms of vector bundles established in Section \ref{seclagmor}. We will also  establish the following
 \begin{theo} \label{theopourrankloci} Let  $X$ be a complex manifold of dimension $2n$ equipped with a symplectic holomorphic structure and let $f:X\rightarrow B$ be a holomorphic Lagrangian fibration, with $B$ smooth. Then for any irreducible component $X_{k,i}$ of $X_k$, with image $B_{k,i}$ in $B$, the following hold

 (1) The dimension  ${\rm dim}\,B_{k,i}$ is at most $ k$.

 (2) The relative dimension of $X_{k,i}$ over $B_{k,i}$ is at least $k$.

 (3) If we have equality in (1) and (2), any connected component of the  general fiber $X_{k,i,b}$, $b\in B_{k,i}$, is a complex torus $T$ of dimension $k$.

 (4) Under the same assumptions as in (3), the locally closed complex submanifold  $X_k$ is smooth along $X_{k,i,b}$ and the restricted normal bundle $N_{X_{k}/X\mid X_{k,i,b}}$ is on each connected component $T$ a homogeneous vector bundle on the complex torus $T$.
 \end{theo}
  Furthermore, we prove in Theorem  \ref{propnewauckland} that the  dimension of the locus $X_k$ where $f$ has rank $k$ is expected to be at least $2k$ and the dimension of its image $B_k$ in $B$ is expected to be at least $k$, unless some unexpected cohomological  vanishing (\ref{eqvanspecial}), (which is  stronger than  (\ref{relverbchern})),  holds.
 As a consequence of our arguments, we get a positivity result for the Chern classes which provides some evidence for the  questions asked in \cite{niep}, see also \cite{osv}.
\begin{theo} \label{theochitop} (Cf. Theorem \ref{propnewauckland}) Let $f: X\rightarrow B$ be a lagrangian fibration of a hyper-K\"{a}hler $2n$-fold with Lagrangian line bundle $L$. Assume the base $B$ is smooth and the rank loci $X_{k}$ satisfy  ${\rm dim}\,X_k= 2k$, ${\rm dim}\,f(X_k)=k$ for all $k$.
Then the classes $L^{n-i}(X)c_{2i}(X)$ are $\mathbb{Q}$-effective.
In particular, for $i=n$, one has $\chi_{\rm top}(X)\geq 0$, and furthermore $\chi_{\rm top}(X)>0$ if the locus $X_0$ where the rank of $f$ is $0$ is not empty.
\end{theo}
The last  result is easy and follows from Beauville's argument in \cite{beauville} if we know that the group scheme
$\mathcal{G}$,   which is constructed under some assumptions in \cite{arinkinfedorov}, \cite{kim} and \cite{kim2}, acts on $X$.
\begin{rema} {\rm The question of the existence of the relative group scheme $\mathcal{G}$ over the base $B$,  acting on $X$ over $B$ (extending the relative Albanese variety defined over the regular locus of $f$ and acting by translation on the smooth fibers) seems to be  open when some fibers are reducible or nonreduced. One of our motivations in this paper is   to analyze the geometry of Lagrangian fibrations without assuming it.}
\end{rema}
\begin{rema} {\rm The numerical positivity of the class $c_2(X)$ is well-known and follows from the L\"ubke (or Bogomolov-Miyaoka-Yau) inequality and existence of K\"ahler-Einstein metrics.}
\end{rema}

  In Section \ref{secsupport} we will   establish  Conjecture \ref{conjvanop}  assuming that the base of the Lagrangian fibration is smooth and under a dimension  assumption on the rank strata $X_k$.
\begin{theo}\label{theonewplustechass} Let $f: X\rightarrow B$ be a Lagrangian fibration and $i$ be a positive integer. Assume that $B$ is smooth in codimension $i-1$ and that for any
$k\geq n-i+1$, and any irreducible component $Z$ of $X_k$, we have either ${\rm dim}\,Z\leq 2k$ or  ${\rm dim}\,f(Z)\geq k$.
Then  for $j\geq i$, $c_{2j}(X)$ vanishes in ${\rm CH}(X\setminus f^{-1}(B^i))$ for some codimension $i$ closed algebraic subset $B^i$ of $B$.
\end{theo}

\vspace{0.5cm}

{\bf Thanks.} {\it  This paper was finished during my stay in Auckland as a Michael Erceg Senior Visiting Fellow. I thank the   Margaret and John Kalman Charitable Trust for its generous support and the  Mathematics department of Auckland University  for this invitation and excellent working conditions. I also thank St\'{e}phane Druel for answering my questions and the referee for a very careful reading and helpful comments.}
\section{Proof of Theorem \ref{theonew} \label{sectheonew}}
\begin{proof}[Proof of Theorem \ref{theonew}(i)] Let  $f: X\rightarrow B$ be a Lagrangian fibration, where $X$ is projective hyper-K\"{a}hler of dimension $2n$.  Let $H$ be a very ample line bundle  on $B$, whose pull-back to $X$ is thus a multiple  $\mu L$. Let $i\leq n$ be a positive integer.  We will prove the vanishing
$$c_{2j}(\Omega_{X})L^{n-i+1}=0\,\,{\rm in}\,\,{\rm CH}(X)$$ for $j\geq i$  by   induction  on $i$.
 By Bertini,   a  general set of $n-i+1$ sections of $f^*H$ on $X$ defines a smooth complete intersection $\Sigma\subset X$. Along $\Sigma$, we thus have an injective morphism of vector bundles
$$\phi: N_{\Sigma/X}^*\rightarrow \Omega_{X\mid \Sigma},$$
where $N_{\Sigma/X}^*\cong \mathcal{O}_\Sigma(-H)^{n-i+1}$.

As $X$ is hyper-K\"ahler, the vector bundle $\Omega_X$ carries an everywhere nondegenerate skew-symmetric pairing $\langle\,,\,\rangle$.
 We observe that, as $f$ is Lagrangian, the image of $\phi$ is totally isotropic for $\langle\,,\,\rangle$. Indeed, it suffices to prove the result at  a generic point $x$  of $\Sigma$, whose image $b=f(x)$ does not belong to ${\rm Sing}\,B$ and where $f$ is of maximal rank. Then  the image of $\phi$  is contained in $f^*\Omega_{B,b}$, which is Lagrangian in $\Omega_{X,x}$ because $f$ is a Lagrangian fibration.

The subbundle $({\rm Im}\,\phi)^{\perp}$ of $\Omega_{X\mid \Sigma}$ is thus of rank $n+i-1$   and contains ${\rm Im}\,\phi$. Furthermore, the quotient $\Omega_{X\mid \Sigma}/({\rm Im}\,\phi)^{\perp}$ is isomorphic to $N_{\Sigma/X}=N_{\Sigma/X}^{**}$ by duality using  $\langle\,,\,\rangle$. Let
 \begin{eqnarray}\label{eqcalEnew} \mathcal{E}:= ({\rm Im}\,\phi)^{\perp}/{\rm Im}\,\phi.\end{eqnarray}
 This is a vector bundle of rank $2i-2$ on $\Sigma$, hence we have
 \begin{eqnarray}\label{eqannurank} c_{2i}(\mathcal{E})=0\,\,{\rm in}\,\,{\rm CH}(\Sigma).\end{eqnarray}
We now use the exact sequences
$$0\rightarrow N_{\Sigma/X}^* \stackrel{\phi}{\rightarrow} ({\rm Im}\,\phi)^{\perp}\rightarrow \mathcal{E}\rightarrow 0,$$
$$0\rightarrow ({\rm Im}\,\phi)^{\perp}\rightarrow \Omega_{X\mid\Sigma}\rightarrow N_{\Sigma/X}\rightarrow 0$$
explained above and the Whitney formula, which gives equalities in  ${\rm CH}(\Sigma)$
\begin{eqnarray}\label{eqW1} c(\Omega_{X\mid \Sigma})=c(N_{\Sigma/X}) c(({\rm Im}\,\phi)^{\perp})\\
\nonumber
=c(N_{\Sigma/X})c(N_{\Sigma/X}^*) c(\mathcal{E}).
\end{eqnarray}
We also get by inverting the total Chern classes in (\ref{eqW1})
\begin{eqnarray}\label{eqW2} c(\mathcal{E})=c(\Omega_{X\mid \Sigma})c(N_{\Sigma/X})^{-1}c(N_{\Sigma/X}^*)^{-1}\,\,{\rm in}\,\,{\rm  CH}(\Sigma).
\end{eqnarray}
If we expand the equality $c(\Omega_{X\mid \Sigma})=c(N_{\Sigma/X})c(N_{\Sigma/X}^*) c(\mathcal{E})$ of  (\ref{eqW1}), and take into account the fact that
$c_{2i}(\mathcal{E})=0$ and $c(N_{\Sigma/X})=(1+\mu L_{\mid \Sigma})^{n-i+1}$, we find that
\begin{eqnarray}\label{eqdevc2i} c_{2i}(\Omega_{X\mid \Sigma})=\sum_{l>0} \alpha_l L_{\mid \Sigma}^{2l} c_{2i-2l}(\mathcal{E})
\end{eqnarray}
for some integers $\alpha_l$.
Using (\ref{eqW2}), we can replace in (\ref{eqdevc2i}) the terms $c_{2i-2l}(\mathcal{E})$ by polynomials in $L$ and $c_k(\Omega_{X\mid \Sigma})$, which provides us with an equation
\begin{eqnarray}\label{eqdevc2ireplace} c_{2i}(\Omega_{X\mid \Sigma})=\sum_{l>0} \beta_l L_{\mid \Sigma}^{2l} c_{2i-2l}(\Omega_{X\mid \Sigma})
\end{eqnarray}
for some integers $\beta_l$.
Using the fact that the class of $\Sigma$ in $X$ is a nonzero multiple of $L^{n-i+1}$ in ${\rm CH}(X)$, equation (\ref{eqdevc2ireplace}) gives a relation
\begin{eqnarray}\label{eqdevc2ireplacedansX} c_{2i}(\Omega_{X})L^{n-i+1}=\sum_{l>0} \gamma_l L^{2l+n-i+1} c_{2i-2l}(\Omega_{X})
\end{eqnarray}
for some rational numbers $\gamma_l$.

The proof of the vanishing $c_{2i}(\Omega_{X})L^{n-i+1}$ in ${\rm CH}(X)$ thus follows by   induction  on $i$.  The proof of the vanishing of $c_{2j}(\Omega_{X})L^{n-i+1}$ in ${\rm CH}(X)$ for $j\geq i$ works in the same way.
\end{proof}

\section{Lagrangian morphisms of vector bundles \label{seclagmor}}
We study in this section general properties of Lagrangian morphisms of vector bundles introduced in Definition \ref{defimorlag}. We do not know if the  notion has been classically studied.
Our  motivation for introducing this notion and its  relevance for the subject of this paper come from the following
\begin{ex}\label{exlag} Let $X$ be hyper-K\"ahler and let $f: X\rightarrow B$ be a Lagrangian fibration. Then denoting by $B'\subset B$ the smooth locus of $B$ and $X'\subset X$ its inverse image in $X$, the morphism
\begin{eqnarray}\label{eqpullback} \phi:=f^*: f^*\Omega_{B'}\rightarrow \Omega_{X'}
\end{eqnarray}
is Lagrangian on $X'$.
\end{ex}
Note that it is conjectured that, if  $B$ is normal (which we will always assume) $B$ is smooth (and if it is smooth, it is known to be isomorphic to $\mathbb{P}^n$ by \cite{hwang}). In any case, as $B$ is normal, the codimension of $B\setminus B'$ in $B$ is at least $2$, while the discriminant locus,  over which $\phi$ is not everywhere of maximal rank, has codimension $1$ in $B$, so that the notion is already interesting over $B'$.

To study  the rank loci for a general Lagrangian morphism, we consider the universal situation, where the basis is the set  $M_{\rm lag}$ of  matrices $M$  of size $(2n,n)$ which are Lagrangian, in the sense that the morphism $M:\mathbb{C}^n\rightarrow \mathbb{C}^{2n}$ is Lagrangian.
\begin{lemm} \label{lewarmrank} (i)  The subvariety  $M_{\rm lag}\subset M_{2n,n}$ is an irreducible, local complete intersection  subvariety of codimension $n(n-1)/2$, hence of dimension $2n^2-\frac{n(n-1)}{2}$.

(ii) $M_{\rm lag}$ is  smooth at a matrix $M$ which is of rank $\geq n-1$, but it is singular at a matrix which is of rank $\leq n-2$. Along the set $M_{\rm lag,n-2}$ of  matrices of rank $n-2$, $M_{\rm lag}$ has an ordinary quadratic singularity of codimension $7$.

(iii) The locus $M_{{\rm lag}, n-i}$ of isotropic matrices of rank $n-i$ is smooth and  has codimension $\frac{3i^2+i}{2}$ in $M_{\rm lag}$. In particular, the locus $M_{{\rm lag}, n-1}$ is smooth  of  codimension $2$ in $M_{\rm lag}$ and is contained in its smooth locus, and $M_{{\rm lag}, n-2}$ is smooth  of  codimension $7$ in $M_{\rm lag}$ and is contained in its singular locus.
\end{lemm}
\begin{proof} (i)  Denoting $m_k$ the $k$-th column vector of a matrix $m$, the equations defining $M_{\rm lag}$ in $M_{2n,n}$ are
\begin{eqnarray}\label{eqisot} \langle m_k,\,m_l\rangle=0\end{eqnarray}
for $n\geq l>k\geq 1$. This locus  is thus defined by $n(n-1)/2$ equations. In order to prove that $M_{\rm lag}$ is a local complete intersection,   it suffices to show that its dimension is $2n^2-\frac{n(n-1)}{2}$, which is done by  proving  (iii) since $M_{\rm lag}$ is stratified by the rank $n-i$ of the matrix. The locus of rank $n-i$ isotropic matrices is homogeneous under the symplectic group ${\rm Sp}(2n)$ and of dimension $$(n-i)(n+i)-(n-i)(n-i-1)/2+n(n-i)=2n^2-\frac{n(n-1)}{2}-   \frac{3i^2+i}{2},$$ since a rank $n-i$ isotropic matrix of size $n\times 2n$ determines its image $W_{n-i}$ which is an isotropic vector subspace of $ \mathbb{C}^{2n}$ of dimension $n-i$, and a matrix $\mathbb{C}^n\rightarrow W_{n-i}$ of size $n\times n-i$ and rank $n-i$.  This dimension count proves the codimension statement in  (iii). The irreducibility follows from the above dimension count and the fact that the rank $n$ stratum is homogeneous, hence irreducible.

 (ii) and (iii)  Denote by $m_{ij}$  the entries of a matrix $m\in M_{2n,n}$, where $i\in\{1,\ldots,2n\}$ and $j\in\{1,\ldots, n\}$, and by $m_i$ its $i$-th column vector. Let  $M\in M_{\rm lag}$ be any Lagrangian matrix. As  the image of $M$ is isotropic, we can assume by choosing an adequate basis of $\mathbb{C}^{2n}$ in which   the intersection form $\langle\,,\,\rangle$ is the standard one,
 that for $k\leq {\rm rk}\,M$, the  vector $M_k$ is the basis vector  $f_k$ of $\mathbb{C}^{2n}$, and that it is $0$ for $k>{\rm rk}\,M$.
 For a matrix $m\in M_{2n,n}$, we then write  $m_k=f_k+h_k$ for $k\leq {\rm rk}\,M$, so that the  column vector $h_k(m)$ vanishes at the point $M$. For $k>{\rm rk}\,M$, we write $m_k=:h_k$ and the vector $h_k(m)$ vanishes at the point $M$.
 In the coordinates $h_{kl}$ on $M_{2n,n}$ centered at $M$, the equations (\ref{eqisot}) are
 \begin{eqnarray}\label{eqisotdevelop}  h_{k+n,l}-h_{l+n,k}+ \langle h_k,\,h_l\rangle=0\,\, {\rm for} \,\,k,\,l\leq {\rm rk}\,M\\
 \label{eqisotdevelop1}
 h_{l,k+n} +\langle h_k,\,h_l\rangle=0 \,\, {\rm for} \,\,k\leq {\rm rk}\,M,\,l> {\rm rk}\,M\\
 \label{eqisotdevelop2}
  \langle h_k,\,h_l\rangle=0 \,\, {\rm for} \,\,k> {\rm rk}\,M,\,l> {\rm rk}\,M.
 \end{eqnarray}
 where the terms $\langle h_k,\,h_l\rangle$ are  quadratic in the coordinates $h_{kl}$.

 When the rank of $M$ is $n$, only the equations in the first line (\ref{eqisotdevelop}) appear. When the rank of $M$ is  $n-1$, only the equations in the first and second lines (\ref{eqisotdevelop}), (\ref{eqisotdevelop1}) appear. In both cases, the  linear parts of these equations  are  obviously independent, which proves the result in that case. When the  rank of $M$ is $n-2$, there is one quadratic equation in the third  line (\ref{eqisotdevelop2}), namely $\langle h_{n-1},\,h_n\rangle=0 $. (ii) and the last statement in (iii) follows by studying the rank of this quadratic equation.
\end{proof}
\begin{rema}{\rm  The variety $M_{\rm lag}$ carries a universal Lagrangian morphism, but, as we just proved that it is singular, we cannot say  that it is universal for the study of Lagrangian morphisms of vector bundles  on {\it smooth } bases $Y$. However, as $M_{{\rm lag},\geq n-1}$ is smooth, it is universal for the study of Lagrangian morphisms of rank $\geq n-1$ of vector bundles of ranks $n,\,2n$  on smooth bases $Y$.}
\end{rema}
\begin{rema}\label{remageoranksup} {\rm Consider the case of Example \ref{exlag}. Assume for simplicity that the singular fibers have normal crossing singularities. Then the locus $Z$ where $\phi=f^*$ has corank $1$ is the union of the singular loci of the fibers which is expected to have codimension $2$, by  Lemma \ref{lewarmrank}, so that singular  fibers are singular in codimension $1$. If furthermore the normalization of the locus $Z$ is smooth of dimension $2n-2$,  the locus  where $\phi=f^*$ has corank $i$ is the locus of $i$-branches singularities in the fibers, which is expected to appear in codimension $2i$ (being the locus of intersection of $i$ branches of $Z$). So the geometric dimension count does not fit with the abstract dimension count of Lemma \ref{lewarmrank}. We will discuss from a different viewpoint  this lack of transversality in Section \ref{secranklocilagfib}. }
\end{rema}
In the rest of this section, we will  focus on the first rank stratum parameterizing matrices of rank $n-1$, which is the only one to be studied in order to establish the vanishing (\ref{relchernclassverbchowavecsquare}) because, by Sard's theorem, the locus where a Lagrangian map $f:X\rightarrow B$  has rank $\leq n-2$ maps to a locus of dimension $\leq n-2$ in the base $B$.
Our first result   is the following (cf. Theorem \ref{theonewlag}).
\begin{prop}\label{pronullC} Let $\phi: F\rightarrow E$ be a Lagrangian morphism on a smooth variety $Y$, with ${\rm rank}\,E=2n$.  Assume that

(*)  the codimension of the locus $Y_{\leq n-1}\subset Y$ where $\phi$ has rank $\leq n-1$ is at least $ 2$.

Then

(i)  $Y_{n-1} $ is  a local codimension $2$ complete intersection in $Y$.

(ii) There is an exact sequence on $Y$
\begin{eqnarray}\label{eqexact}0\rightarrow F\stackrel{\phi}{\rightarrow} E\stackrel{\phi^*}{\rightarrow} F^*\rightarrow \mathcal{G}\rightarrow 0,
\end{eqnarray}
where $\mathcal{G}$ is supported on $Y_{\leq n-1}$ and is a line bundle on $Y_{n-1}=Y_{\leq n-1}\setminus Y_{\leq n-2}$.

(iii) The normal bundle of $Y_{n-1}$ in $Y$ is isomorphic to $\mathcal{H}om(\mathcal{E}_2,\mathcal{G})$, where $\mathcal{E}_2$ is a rank $2$ vector bundle with trivial determinant on $Y_{n-1}$.
\end{prop}

In (\ref{eqexact}), $\phi^*$ denotes the transpose of $\phi$, which is defined using the self-duality of  $E$ given  by $\langle\,,\,\rangle$.
\begin{proof} Statement (i) follows,  by local trivialization of the vector bundles $F$ and $E$ (equipped with its symplectic structure), from Lemma \ref{lewarmrank}, (ii) and (iii), which say that $M_{{\rm lag},n-1}$ is a smooth codimension $2$ locally closed subvariety of $M_{\rm lag}$ and is  contained in its smooth locus, hence is a local complete intersection of codimension $2$.

(ii) As the image of $\phi$ is isotropic, we have $\phi^*\circ \phi=0$. Furthermore, at any point where $\phi$ has rank $n$, we  have ${\rm Ker}\,\phi^*=({\rm Im}\,\phi)^{\perp}={\rm Im}\,\phi$ hence the sequence
\begin{eqnarray}\label{eqencoreex} 0\rightarrow F\stackrel{\phi}{\rightarrow} E\stackrel{\phi^*}{\rightarrow} F^*\end{eqnarray}
is exact on the right and at the middle on the Zariski open set $Y\setminus Y_{\leq n-1}$ of $Y$. Under  assumption (*), this implies that the sequence remains exact at the middle everywhere on $Y$. Indeed, let $\alpha$ be a local section of ${\rm Ker}\,\phi^*$ defined on a Zariski open set $U$ of $Y$. Then, on $U\setminus (Y_{\leq n-1}\cap U)$, $\alpha=\phi(\beta)$ by the exactness of the sequence (\ref{eqencoreex}) on $U \setminus (Y_{\leq n-1}\cap U)$. As $U$ is smooth and $Y_{\leq n-1}\cap U$ has codimension $\geq 2$ in $U$, $\beta $ extends to a section $\tilde{\beta}$ of $F$ on $U$ and
$\alpha=\phi(\tilde{\beta})$. This proves the exactness of (\ref{eqexact}) in the second term. Finally, the cokernel of $\phi^*$ is a line bundle supported on $Y_{n-1}$,  because   it is isomorphic to the cokernel of the rank $n-1$ morphism
$$ E_{\mid Y_{n-1}}\stackrel{\phi^*_{\mid Y_{n-1}}}{\rightarrow} F^*_{\mid Y_{n-1}}$$
of vector bundles on $Y_{n-1}$.

(iii) The vector bundle $\mathcal{E}_2$ on $Y_{n-1}$ is defined as follows: the morphism $\phi_{\mid  Y_{n-1}}$ has rank $n-1$, hence its image in ${\rm Im}\,(\phi_{\mid  Y_{n-1}})\subset E_{\mid Y_{n-1}}$ is a vector subbundle of rank $n-1$, which is totally isotropic for $\langle\,,\,\rangle$. We define $\mathcal{E}_2$ as $({\rm Im}\,\phi)^{\perp}/{\rm Im}\,\phi={\rm Ker}\,\phi^*/{\rm Im}\,\phi$. The fact that $\mathcal{E}_2$ has trivial determinant follows from the fact $\langle\,,\,\rangle$ induces a nondegenerate skew-symmetric pairing on $\mathcal{E}_2$, which trivializes its determinant. It remains to prove that the normal bundle of $Y_{n-1}$ in $Y$ is isomorphic to $\mathcal{H}om(\mathcal{E}_2,\mathcal{G})$. We first do the case where $Y_{n-1}$ is smooth of codimension $2$. At any point $y\in Y_{n-1}$, we have a well-defined morphism
$$d\phi_y:T_{Y,y}\rightarrow {\rm Hom}({\rm Ker}\,\phi^*_y,{\rm Coker}\,\phi^*_y)$$
whose kernel is the tangent space to $Y_{n-1}$ at $y$. It remains to see that the image is contained in
$${\rm Hom}({\rm Ker}\,\phi^*_y/{\rm Im}\,\phi_y,{\rm Coker}\,\phi^*_y)={\rm Hom}(\mathcal{E}_{2,y},\mathcal{G}_y).$$
This follows immediately from the fact that $\phi^*\circ \phi=0$ on $Y$. As ${\rm Hom}(\mathcal{E}_{2,y},\mathcal{G}_y)$ has dimension $2$,  we get from this construction  a canonical isomorphism $N_{Y_{n-1}/Y}\cong \mathcal{H}om(\mathcal{E}_{2},\mathcal{G})$ in the case where $Y_{n-1}\subset Y$ is smooth of codimension $2$. The general case follows by local trivializations using Lemma \ref{lewarmrank}(ii), (iii).
\end{proof}

\begin{coro}\label{corocrux}  Under the same assumptions as in Proposition \ref{pronullC}, consider the following combination $C$ of total  Chern classes $$C=c(E)c(F)^{-1}c(F^*)^{-1}.$$ Then, if the line bundle $\mathcal{G}$ is trivial on  $Y_{n-1}$, $C_i$ vanishes in  ${\rm CH}(Y\setminus Y_{\leq n-2}) $ for $i\geq 3$.
\end{coro}

\begin{proof} Indeed, by the exact sequence (\ref{eqexact}), one has $C=s(\mathcal{G})$, where $s$ denotes the total Segre class. As $\mathcal{G}$ is trivial on $Y_{ n-1} $ by assumption,  and $Y_{n-1} $ is a local complete intersection of codimension $2$ by Proposition \ref{pronullC}(i),  the normal bundle $N_{Y_{n-1}/Y}$ has trivial determinant by Proposition \ref{pronullC}(i).  Using the fact that $\mathcal{G}\cong \mathcal{O}_{Y_{n-1}}$, the corollary follows  from  Lemma \ref{legenCS} below applied  to the closed subvariety $Y_{n-1}\subset Y\setminus Y_{\leq n-2}$.
\end{proof}
\begin{lemm}\label{legenCS} Let $Y$ be a smooth variety and $Z\subset Y$ be a local complete intersection of codimension $2$. Assume that the determinant of $N_{Z/Y}$ is trivial. Then  the total Segre class $s(\mathcal{O}_{Z})\in{\rm CH}(Y)$   satisfies $s_i=0$ for $i\geq 3$.
\end{lemm}
\begin{proof} The obvious case is when $Z\subset Y$ is the zero locus of  a section of a rank $2$ vector bundle $M$ on $Y$ with trivial determinant. Indeed, we have in this case the two exact sequences
\begin{eqnarray} \label{eq1triv} 0\rightarrow \mathcal{I}_Z\rightarrow \mathcal{O}_Y\rightarrow \mathcal{O}_Z\rightarrow 0,\\
\label{eq2triv}
0\rightarrow \mathcal{O}_Y\rightarrow M\rightarrow \mathcal{I}_Z\rightarrow 0.
\end{eqnarray}
From (\ref{eq1triv}), we deduce that $s(\mathcal{O}_Z)=c(\mathcal{I}_Z)$ and from (\ref{eq2triv}), we deduce that $c(\mathcal{I}_Z)=c(M)$. Finally $c_i(M)=0$ for $i\geq 3$ since $M$ has rank $2$.

The general case follows from this case. There exists a rank $2$ vector bundle $M$ on $Y$ with the desired property if the natural morphism
\begin{eqnarray}\label{eqexactlocalglobal} {\rm Ext}^1(\mathcal{I}_Z,\mathcal{O}_Y)\rightarrow H^0(Z,\mathcal{E}xt^1(\mathcal{I}_Z,\mathcal{O}_Y))=H^0(Z,\mathcal{O}_Z)\end{eqnarray}
 is surjective. Here the isomorphism
$\mathcal{E}xt^1(\mathcal{I}_Z,\mathcal{O}_Y)\cong \mathcal{O}_Z$
is given by  the triviality of the determinant of $N_{Z/Y}$.
 The surjectivity of the local to global map (\ref{eqexactlocalglobal}) is satisfied  if $Y$ is affine, hence the lemma is proved in the affine case. If $Y$ is not affine, we can use the Jouanolou trick \cite[Lemme 1.5]{jouanolou}  and replace $Y$  by an  affine bundle $A_Y$ over $Y$, which has isomorphic Chow groups since it is fibered into affine spaces over $Y$ and whose total space is affine. The result then applies to $A_Z\subset A_Y$.
\end{proof}

\subsection{The case of Lagrangian fibrations}
We study in this section the Lagrangian morphism of vector bundles given in Example \ref{exlag}, namely $\phi=f^*$ for some Lagrangian fibration $f:X\rightarrow B$, where $B$ is normal. More precisely, we will consider the restriction $f:X'\rightarrow B'$, where $B'$ is the smooth locus of $B$ and $X':=f^{-1}(B')$. As before, we denote by $X'_i$ the locally closed algebraic subset of $X'$ where $f^*$ has rank $i$. We will  give  a  second  proof of Theorem \ref{theonew} for $i=2$, based on the following

\begin{prop}\label{propfacilechouette}
 There exists a Zariski open set  $B''\subset  B'$ such that $B\setminus B''$ has codimension $\geq 2$ in $B'$ (hence in $B$) and contains  $f(X'_{n-2})$,  with the following property. Let   $X'':=f^{-1}(B'')$, so $\phi$ has rank $\geq n-1$ everywhere on $X''$ since $X''\cap X'_{n-2}=\emptyset$. Then the  line bundle $\mathcal{G}$ on $X''_{n-1}$ constructed in  Proposition \ref{pronullC}(ii)  is trivial along the fibers of $f_{\mid X''_{n-1}}: X''_{n-1}\rightarrow B''$.
\end{prop}
\begin{proof}
 Let $\tau:\widetilde{X}'_{n-1,{\rm red}}\rightarrow X'_{n-1}$ be a desingularization of the locus $X'_{n-1}$ equipped with its reduced structure and let $$\tilde{f}_{n-1}:=f\circ \tau:\widetilde{X}'_{n-1,{\rm red}}\rightarrow B'.$$ Using  Sard's theorem, the locus $Z\subset\widetilde{X}'_{n-1,{\rm red}}$ where $\tilde{f}_{n-1}$ has rank $<n-1$  satisfies ${\rm dim}\,\tilde{f}_{n-1}(Z)\leq n-2$. Similarly, the closed algebraic subset $f(X'_{\leq n-2})$ of $B$  has dimension $\leq n-2$.
Let $B'_1:= B'\setminus (f(X'_{\leq n-2})\cup \tilde{f}_{n-1}(Z))$.
By construction, $\tilde{f}$ has rank exactly $n-1$ on $\widetilde{X}_{n-1,{\rm red}} \cap \tilde{f}^{-1}(B'_1)$. Its image is thus locally a finite union of smooth hypersurfaces in $B'_1$ and, up to removing a codimension $2$ closed algebraic subset of $B'_1$ where this hypersurface has at least two branches, we get the desired Zariski open set $B''$. By construction, the hypersurface $\Delta= \tilde{f}(\widetilde{X}_{n-1,{\rm red}} \cap \tilde{f}^{-1}(B''))$ is a  smooth hypersurface in $B''$ and
$\tilde{f}:\widetilde{X}_{n-1,{\rm red}} \cap \tilde{f}^{-1}(B'')\rightarrow \Delta$ is everywhere of maximal rank $n-1$. Having constructed $B''$, it remains to prove the statement about $\mathcal{G}$.
We now observe that, for any point $y$ in  the fiber $F_b$ of $\tilde{f}$ over $b\in \Delta$, the image of $f_*:T_{X,\tau(y)}\rightarrow T_{B,b}$ has dimension $n-1$ and contains the image of $\tilde{f}_*:T_{\widetilde{X}_{n-1,{\rm red}},y}$, namely $T_{\Delta,b}$. Hence the two spaces are equal and we find that the image of $f_*:T_{X,\tau(y)}\rightarrow T_{B,b}$ is constant, equal to $T_{\Delta,b}$  along $F_b$. This shows that $\mathcal{G}$ is trivial along the fiber $F_b$. This is a priori not enough to prove the desired statement since we worked only with a desingularization of the reduced structure of ${X}_{n-1}$ and not with ${X}_{n-1}$ itself.
To complete the proof, we use the following
\begin{lemm}\label{leidiot} Let $Z$ be projective scheme and $L$ be a line bundle on $Z$. Assume that

(i) $L$ is generated by global sections.

(ii) The pull-back of $L$ to a desingularization $\widetilde{Z}_{\rm red}$ of the underlying reduced subscheme $Z_{\rm red}$ is trivial.

Then $L$ is trivial.
\end{lemm}
\begin{proof} As $L $ is generated by global sections, it induces a morphism $g:Z\rightarrow \mathbb{P}^N$ such that $L$ is isomorphic to $g^*\mathcal{O}_{\mathbb{P}^N}(1)$. As the pull-back of $L$ to  $\widetilde{Z}_{\rm red}$ is trivial, this morphism has image supported on a finite number of points. Any line bundle  on a $0$-dimensional scheme is trivial and this finishes the proof.
\end{proof}
Lemma \ref{leidiot} applies in our situation since the line bundle $\mathcal{G}$ on $X'_{n-1}$ is by definition a quotient of the vector bundle $f^*\Omega_{B'}^*$, hence is generated by global sections along the fibers of $f:X'_{n-1}\rightarrow B'$.
\end{proof}
Combining the results of the previous sections, we get another proof   of Theorem \ref{theonew}(i) for $i=2$.
\begin{coro} \label{thoeversionfaible} Let $f:X\rightarrow B$ be a Lagrangian fibration of a hyper-K\"{a}hler manifold, with $B$ normal. Let $L$ be the Lagrangian line bundle. Assume that there is no divisor in $X$ made of non-reduced fibers. Then for $i\geq 2$, the Chern classes $c_{2i}(X)\in {\rm CH}^{2i}(X)$ are supported over a closed algebraic subset of codimension $\geq 2$ in $B$. In particular
\begin{eqnarray} \label{eqnewauckland}  L^{n-1}c_i(X)=0\,\,{\rm in}\,\,{\rm CH}(X)\,\,{\rm for}\,\,i\geq 4.
\end{eqnarray}
\end{coro}
\begin{proof} Let $B''$ and $X''$ be as in Proposition \ref{propfacilechouette}.  By our assumption on singular fibers, the locus $X''_{n-1}=X''_{\leq n-1}\subset X''$ where $f: X''\rightarrow B''$ has rank $ n-1$ is a closed algebraic subset of  codimension $2$ in $X''$. Proposition \ref{propfacilechouette} tells that the line bundle $\mathcal{G}$ on $X''_{n-1}$ is trivial on the fibers of $f_{n-1}:X''_{n-1}\rightarrow B''_{n-1}$, so it is trivial on a Zariski open set $f_{n-1}^{-1}(U)$ for some dense Zariski open set $U:=B''_{n-1}\setminus W$ of $B''_{n-1}$. Let $B''':=B''\setminus W$ and $X''':=f^{-1}(B''')$. Then $B'''\subset B$ is the complement of a closed algebraic subset of codimension $\geq 2$. Furthermore, by construction,  the line bundle $\mathcal{G}$  is trivial on $X'''_{n-1}$. We can thus apply Corollary \ref{corocrux} to conclude that $c_{2i}(\Omega_X)$ vanishes in ${\rm CH}(X''')$ for $i\geq 2$.
\end{proof}
\begin{rema}{\rm We will give a different proof of Corollary \ref{thoeversionfaible} in Section \ref{secsupport}. More precisely, Corollary \ref{coroideux} reproves  Corollary \ref{thoeversionfaible} without  the assumption on non-reduced fibers.}
\end{rema}

\subsection{Non-reduced  fibers in codimension 1 \label{secmultiple}}
We show in this section how to get rid of the assumption that there are no non-reduced fibers in codimension $1$. Our goal is to show how, in the case of a Lagrangian fibration of a hyper-K\"ahler manifold of Picard number $2$, we can modify  the morphism (\ref{eqpullback}) over the complement of a closed algebraic subset of codimension $\geq 2$ of the base so as to get a Lagrangian morphism of vector bundles having the property that its locus $X_{n-1}$ of jumping rank has codimension at least $2$, and to which the arguments of the previous section apply.
Note that the existence of a divisor in the base parameterizing non-reduced fibers can occur, as in the following
\begin{ex} \label{exK3} Consider an elliptic  $K3$ surface $f:S\rightarrow \mathbb{P}^1$ with elliptic general fiber $E_t,\,t\in \mathbb{P}^1$. Then for any $n\geq 2$, we have a composed morphism
\begin{eqnarray}f_n: S^{[n]}\rightarrow S^{(n)}\rightarrow (\mathbb{P}^1)^{(n)}=\mathbb{P}^n,
\end{eqnarray}
which gives a Lagrangian fibration with non-reduced fibers over the big diagonal of $(\mathbb{P}^1)^{(n)}$. More precisely, for $n=2$, the  fiber of $f_2$ over $t+t'\in (\mathbb{P}^1)^{(2)},\,t\not=t'$, is isomorphic to  $E_t\times E_{t'}$, and for $t=t'$, the fiber is as a set  the union of  $ E_t^{(2)}$ and another component which is also a $\mathbb{P}^1$-bundle over $E_t$, namely the set of length $2$ subschemes of $S$ supported at one point of $E_t$. The first component is not reduced since near a general point $(x,y),\,x\not= y$ of $E_t^{(2)}$, the morphism $f_2$ locally factors through the morphism $\mathbb{P}^1\times \mathbb{P}^1\rightarrow (\mathbb{P}^1)^{(2)}$ which ramifies over  the diagonal. The second component is reduced because the morphism $f_2:S^{[2]}\rightarrow (\mathbb{P}^1)^{(2)}$ is of rank $2$ at a general point of the divisor parameterizing nonreduced schemes of length $2$. Indeed, considering a general   curve $C\subset S$, the morphism  $f_{2\mid C^{(2)}}: C^{(2)}\rightarrow  (\mathbb{P}^1)^{(2)}$ is  a local analytic  isomorphism at a  point $2c\in C^{(2)}$,  once    $f_{\mid C}: C\rightarrow \mathbb{P}^1$ is a local analytic isomorphism at $c$.
\end{ex}
In this example, there are   non-reduced fibers in codimension $1$ which are not multiple fibers. This however cannot occur when $X$ has Picard number $2$ by the following lemma.

\begin{lemm} \label{rho2projpasdepb}  Let $f: X\rightarrow B$ be a Lagrangian fibration with $X$ projective of Picard number $2$. Then the general non-reduced fibers appearing in codimension $1$ are multiple fibers.
\end{lemm}
\begin{proof} As $B$ is normal, hence smooth in codimension $1$,  the general non-reduced fibers appearing in codimension $1$ appear over the smooth locus $B'$, over which the Lagrangian fibration is flat by Matsushita's theorem \cite{matsuflat}. Assume that over a generic point $t\in B^1\subset B'$ of a divisor in $B'$, the fiber $X_t$ is not reduced but not a multiple fiber. This implies that  it has several irreducible components $X_{t,l}$ appearing with different multiplicities $l$. The inverse image $D:=f^{-1}(B^1)$ thus has  several irreducible components $D_k$,   where $D_k$ is  defined as the union of components of the general fiber  of $f$ over $B^1$  with  multiplicity $k$.  As the fiber $X_t$ is a local complete intersection, its irreducible components meet in codimension $1$ and as it is connected, it follows that $D_k$ is effective and non-trivial when restricted to a component  $X_{t,l}$ for some $l\not=k$. As the divisor $D_k$ is trivial on the general fiber $ X_{t'},\,t'\in B$, we conclude that $\rho(X)\geq 3$, since ${\rm Pic}(X)$ already contains the Lagrangian line bundle, which is trivial on $X_{t,l}$ and $X_{t'}$, and an ample divisor, which is nontrivial on both  $X_{t,l}$ and $X_{t'}$. This contradiction concludes the proof.
\end{proof}
Recall that, by Matsushita \cite{matsushita}, the  Lagrangian fibrations deform over a codimension $1$ locally closed analytic subspace of the  Kuranishi family of $X$, so that, for the general projective deformation $(X_t,f_t)$ of $(X,f)$, $X_t$ has Picard number $2$ and  Lemma \ref{rho2projpasdepb} applies.
We now   consider the   case of a Lagrangian fibration $f:X\rightarrow B$, such that the non-reduced  fibers appearing in codimension $1$, that is, over a divisor $B^1$ of $B$, are  multiple fibers.
We work again  over the Zariski open set  $B''$ of $B'$ where the various components of the divisor $B^1$ do not intersect  and over which $f$ has rank $\geq n-1$. (This is the complement of a codimension $2$ subset in $B$.) We now make the following construction:
As $f: X''\rightarrow B''$  is of rank $n-1$ over ${B''}^1:=B^1\cap B''$, the morphism $\phi:f^*\Omega_{B''}\rightarrow \Omega_{X''}$ can be modified into a morphism
\begin{eqnarray} \label{eqphilog} \phi_{\rm sat}: (f^*\Omega_{B''})_{\rm sat} \rightarrow \Omega_{X''},\end{eqnarray}
which is now generically of rank $n$ along  $D$, where the vector bundle  $( f^*\Omega_{B''})_{\rm sat}$ can be defined over $B''$  as  the saturation of $f^*\Omega_{B''}$ in $\Omega_X$ and has the following explicit description: let $D:=f^{-1}(B^1)$   (equipped with its reduced structure). By our assumption on $B''$, at any point of $X''\cap D$,  the kernel of $\phi:f^*\Omega_{B''}\rightarrow \Omega_{X''}$ over $t\in {B''}^1$ equals $dg_t$ where $g_t$ is a defining equation of $B^1$ near $t$. Indeed, if the fibers over the component of $B^1$ passing through $t$ have multiplicity $k$, then we have locally in a neighborhood of  $f^{-1}(t)$,
$f^* g_t=h_t^k$  for a function $h_t$ on $X$  defining $D$, hence $f^*dg_t=kh_t^{k-1} dh_t$. The vector bundle
$(f^*\Omega_{B''})_{\rm sat}$ is then  locally generated by $f^*\Omega_{B''}$ and $dh_t$, that is,  by $f^*\Omega_{B''}$  and $\frac{f^*dg_t}{h_t^{k-1}}$.

Using this construction and the fact that the Chern classes of $(f^*\Omega_{B''})_{\rm sat}$ are pulled-back from $B''$ with rational coefficients (as is the class of the divisor $D$),  all the vanishing results for Chern classes  obtained in the previous section and relying on  Proposition \ref{pronullC} on Lagrangian morphisms of vector bundles with degeneracy locus of codimension $2$, in particular Corollary \ref{thoeversionfaible}, can be proved without  assuming the non-existence of  non-reduced fibers in codimension $1$. When the Picard number of $X$ is $>2$, we first have to deform $X$ to a general $X_t$ with Picard number $2$, prove the result for $X_t$ using the construction above, and specialize it to $X$ (see \cite{riess}). We do not give the full argument here as we will discuss another approach in the next section.
\section{Higher rank loci for Lagrangian fibrations\label{secranklocilagfib}}
We study in this section the stratification by the rank of a Lagrangian fibration with smooth base.
Let $f:X\rightarrow B$ be a Lagrangian fibration of a hyper-K\"{a}hler manifold $X$  with ${\rm dim}\,X=2n$. We assume that $B$ is smooth (or restrict to its smooth locus $B'$).   For any integer $k$,  we consider  the Zariski  locally closed   subset
$X_{k}\subset X$ of points where $f$ has rank $k$.
\begin{lemm} \label{lemmanewauckland} The relative dimension of $X_k$ over $B$ is at least $k$.
\end{lemm}
\begin{proof} Let $x\in X_k$ and let $b=f(x)\in B$. Let $f_1,\ldots,f_k$ be $k$ algebraic or holomorphic functions on $B$,  defined near $b$, such that the differentials $f^* df_i$ are independent near $x$. Then the Hamiltonian vector fields
$\chi_i$ defined on $X$ near $x$ by the formula
$$\chi_i\lrcorner \sigma_X=f^*df_i$$
are independent near $x$ and commute, since the fibration$f$  is Lagrangian. They thus generate a holomorphic  foliation of an Euclidean neighborhood $V$ of $x$ in $X$ and a  free   action $U\times V\rightarrow V'$ on $V\subset V'$   of a germ $U\subset \mathbb{C}^k$ of  commutative group of automorphisms. As $f$ is a Lagrangian fibration, the considered  automorphisms $\psi\in U$ preserve $f$, namely $f\circ \psi=f$ on $V$ for any $\psi\in U$. It follows that the locus $X_k\cap V$ is preserved by the action of $U$ (that is, mapped to $X_k\cap V'$) and  for any $x'\in X_k\cap V$, the orbit $U\cdot x'$ is contained in $X_k\cap V'$. As these orbits are contained in the fibers of $f_{\mid V}$, the lemma is proved.
\end{proof}
\begin{rema}\label{remagroupaction} {\rm  This  vertical group action  is  also used in \cite{hwangoguiso}. Thanks to the fact that the action is vertical and the map $f$ is proper, the  germ of commutative group appearing  above  globalizes to a  holomorphic action of a commutative group isomorphic to $\mathbb{C}^k$  on a neighborhood of  the fiber $X_t$ in $ X$ passing through $x$. This is obtained by observing that, by properness, the flow generated by Hamiltonian vector fields is defined for  all time. This action  is free assuming that the differentials $f^* df_i$ remain independent everywhere along the fiber. This will be the case in Section \ref{secc2square}, where we will consider the case $k=n-1$.}
\end{rema}
We also note the following basic
\begin{lemm} \label{leridicule} Let $f: Y\rightarrow  S$ be an algebraic or analytic  morphism, where both $Y$ and $S$ are smooth.
With the same notation $Y_k$ as above, for any  irreducible component $Z$  of $Y_k$, one has ${\rm dim}\,f(Z)\leq k$.
\end{lemm}
\begin{proof} Indeed, the generic  rank of $f_{\mid Z}$ is not greater than $k$, hence $f(Z)$ has dimension $\leq k$ by Sard's theorem.
\end{proof}

\begin{proof}[Proof of Theorem \ref{theopourrankloci}] Statement (1) is Lemma \ref{leridicule} and statement (2) is Lemma \ref{lemmanewauckland}. In order to prove (3) and (4), we need to globalize the  argument used in the proof of Lemma \ref{lemmanewauckland}. Let $Z$ be an irreducible  component of $X_k$ which has dimension $2k$ and whose image $f(Z)$ in $B$ has dimension $k$. We observe that,  by applying again Lemma \ref{leridicule} to $X_{\leq k-1}$, for a general point  $b\in f(Z)$, the fiber $Z_b:=(f_{\mid Z})^{-1}(b)$ is contained in $X_k$, that is, the rank of $f$   along $Z_b$ is everywhere $k$.  There is thus a morphism $$\tilde{f}_Z:Z\rightarrow {\rm Grass}(k, f^*T_{B}),$$
$$z\mapsto {\rm Im}\,f_{*,z}\subset T_{B,f(z)},$$
 which is well-defined in a neighborhood of $Z_b$. This morphism is generically constant along $Z_b$, since  at any point of $Z\cap X_k$ where the rank of $f_{\mid Z}$ is also $k$, we have
 $${\rm Im}\,f_{\mid Z, *}={\rm Im}\,f_{*}$$
 and the space ${\rm Im}\,f_{\mid Z, *}$ has to be equal to $T_{f(Z),b}$ at a general point of $f(Z)$. It follows that the morphism $\tilde{f}_Z$ is actually constant along $Z_b$. We now argue as in the proof of Lemma \ref{lemmanewauckland}, in a global setting. Let $b=f(z)$ and let $f_1,\ldots, f_k$ be holomorphic   functions defined on $B$ near $b$ whose restrictions to $f(Z)$  have independent differentials at $b$. By the above argument, the pulled-back functions $g_i=f_i\circ f$ on $X$ have independent differentials along $Z_b$, hence in a neighborhood of $Z_b$ in $X$. Furthermore, their  Poisson brackets vanish  since $f$ is Lagrangian.
It follows that the corresponding Hamiltonian vector fields $\chi_i$ generate an integrable foliation, whose underlying vector bundle is trivial along $Z_b$, and which is vertical, in the sense that it is annihilated by $f_*$, or equivalently that the   diffeomorphisms $\psi_{i,t},\,t\in\mathbb{C}$, generated by the Hamiltonian vector fields $\chi_i$ satisfy $f\circ \psi_{i,t}=f$. Recall from Remark \ref{remagroupaction} that these diffeomorphisms $\psi_{i,t}$ are in fact  defined in a neighborhood of $X_b$ in $X$.
As we have $f\circ \psi_{i,t}=f$, $\psi_{i,t}$ preserves $X_k$, hence it preserves its irreducible components and thus  acts on $Z$ and  on the fiber $Z_b$. As the  group $U\cong \mathbb{C}^k$ generated by the $\psi_{t,i}$'s  is of dimension $k$ and acts freely on $X$, and by assumption ${\rm dim}\,Z_b=k$, we conclude that the orbits of this germ of groups are open in  $Z_b$. In particular $Z_b$ is smooth and its tangent bundle is trivial, isomorphic to the  restriction of $\mathcal{F}$ to $Z_b$, and generated by commuting vector fields. Thus $Z_b$ is a disjoint union of  compact complex tori $T_i$, and the  group action above of $U$ on $Z_b$  factors through the action of  $T_i$ on itself by translations (thus $U$ identifies to the universal cover of $T_i$). As we already observed, $U$   not only  acts  on $T_i$  by translations, but also on a neighborhood of $T_i$ in $X_k$ and $X$. We thus conclude that $U$ acts on the restricted normal bundle $N_{X_k/X\mid T_i}$, making it homogeneous.
\end{proof}

Coming back to Lemma \ref{leridicule}, note that we can easily find examples where  $f(Z)$ has dimension $< k$. For example, consider the case where the morphism $f$ is the blow-up $Y:={\rm Bl}_M(S)\rightarrow S$ of a smooth subvariety $M\subset S$ of codimension $2$. Then the exceptional divisor $E$ is the locus where $f$ is not of maximal rank. The rank of $f$ along $E$ is $m-1$, where $m={\rm dim}\,S={\rm dim}\,Y$, and the rank of $f_{\mid E}$ is $ m-2$. Unlike the Lagrangian fibration maps when the basis is smooth, this example is not flat, but there are also flat examples of this phenomenon: consider a $3$-dimensional singular affine  quadric $Q\subset \mathbb{A}^4$ of equation $x^2+y^2+z^2+t^2=0$ and let $\tau:Y\rightarrow Q$ be a small resolution of the singular point $0\in Q$. Let $f=l\circ \tau: Y\rightarrow \mathbb{A}^2$, where $l:Q\rightarrow  \mathbb{A}^2$ is the linear  projection $(x,y,z,t)\mapsto (x,y)$. Then the fibers of $f$ over $(x,y)\not=0$ are  conics $z^2+t^2=a$ and the fiber of $f$ over $0$ is the union of a conic  $z^2+t^2=0$ and the exceptional curve $E$. They are thus of dimension $1$, so $f$ is flat. The locus where $f$ has rank $1$ is the union of two components: the first component is the set of singular points of the
conics $z^2+t^2=-x^2-y^2$ over the set of points $(x,y)$ such that  $x^2+y^2=0$, and the restriction of $f$ to this component has generic rank $1$. The second component is  the exceptional curve $E$ and the restriction of $f$ to $E$  has generic rank $0$ since it is contracted. Finally,  the rank of $f$ is generically $1$ along $E$, because if instead of $Y$ we consider the blow-up of $Y$ along $E$, that is the blow-up $\tau':\widetilde{Q}\rightarrow Q$ of $Q$ at its singular point, with exceptional divisor  $F$,  then  the rank of $l\circ \tau'$ is generically $1$ along $F$ and the same follows easily for $Y$.
In the case of  a Lagrangian fibration $f:X\rightarrow B$ with smooth $B$,  we do not know if an example as above is possible:
\begin{question} \label{questionnewauckland}  Let   $f:X\rightarrow B$  be a Lagrangian fibration with $X$ hyper-K\"{a}hler and  $B$ smooth  (or the restriction of a Lagrangian fibration over the smooth locus $B'$ of its base). For any integer $k\geq 0$ and  any component $Z$ of the set
$X_{k}\subset X$ of points where $f$ has rank $k$, is it true that  $f(Z)$ has dimension $k$ (or equivalently, that the generic rank of $f_{\mid Z}$ is $k$)?
\end{question}
Note that an affirmative answer to Question \ref{questionnewauckland} would imply, using  Lemma \ref{lemmanewauckland}, that the dimension of $X_k$ (assuming it is non-empty)  is greater than or equal to $2k$, which is different from the generic codimension estimate of  Lemma  \ref{lewarmrank}.

The stratification of $X$ by the rank, that is, by the $X_k$, induces a collection of subsets $f(X_k)=:B_k\subset B$. It   would be interesting to compare these subsets  to another natural stratification on  $B$  related to the topological degeneration of the fibers of $f$.
Assume that the fibers of $f$ are reduced and irreducible. We will say that the fiber $X_b$ has abelian rank $k$ if the Albanese variety of any desingularization $\widetilde{X}_b$ of $X_b$ has dimension $k$.
The following comparison statement can be found   in a slightly different form in \cite[Proposition 5.17]{arinkinfedorov}.
\begin{prop}  Assume that the base  $B$ is smooth,  the fibers of $f$ are reduced and irreducible so  that, by \cite{arinkinfedorov},  there exists a group scheme $\mathcal{G}$ over $B$ with   Lie algebra   isomorphic to $\Omega_B$ which acts on $X$ over $B$. If  a fiber $X_b$, $b\in B$, has abelian rank $k$, then $b$ does not belong to $f(X_l)$ for $l<k$. In other words,  the rank of the differential
$f_*: T_{X,x}\rightarrow T_{B,b}$ is $\geq k$ for any $x\in X_b$.
\end{prop}
\begin{proof} (Cf.   \cite{arinkinfedorov})
The Lie algebra of $\mathcal{G}$ is isomorphic to $\Omega_B$ and the infinitesimal action of $\mathcal{G}$ on $X$ is given at any $x\in X$ by
the pull-back map  $f^*: f^*\Omega_B\rightarrow \Omega_X\cong T_X$ followed by evaluation at $x$. The set $X_l$ of points $x\in X$  where the evaluation map has rank $l$ is thus also the locus where the isotropy subgroup $\mathcal{I}_x$ of $x$ has dimension $n-l$. By our assumption on the singularities of the fibers, any fiber $X_t$ is birational to the corresponding fiber $\mathcal{G}_t$, which is a commutative algebraic group. By  \cite[Theorem 2]{brion}, the group $\mathcal{G}_b$, $b=f(x)$, is an extension of an abelian variety $A_b$ by an affine group $N_b$, so any smooth projective model of  $\mathcal{G}_b$  is  a rationally connected fibration over $A_b$, hence $A_b$ is the Albanese variety of any desingularization  of $\mathcal{G}_b$.  We now observe that, as the fiber $X_{b}$ is irreducible and reduced, there are points $y\in  X_{b}$ such that the isotropy group of $y$ is trivial. It follows first of all that the fiber $X_b$ is birational to the corresponding fiber $\mathcal{G}_b$. Secondly,  considering the action of $\mathcal{I}_x$  on a general point $y\in X_b$ close to $x$, we get that the isotropy group $\mathcal{I}_x$ is affine. It is thus  contained in $N_b$, hence we get, for $x\in X_l$,
 $${\rm dim}\,\mathcal{I}_x=n-l\leq  {\rm dim}\,N_b=n-{\rm dim}\,A_b,$$  which proves the proposition.
\end{proof}

We will  prove in the next section  the following result, which relates Question \ref{questionnewauckland} to unexpected vanishing for Chern classes.
\begin{theo} \label{propnewauckland} Assume that  $B$ is smooth and that for some integer $k\geq 0$, and for  any  $l>n-k$, either $X_l$ is empty or, for any irreducible component $W$ of $X_l$, one has   ${\rm dim}\,W= 2l$ and   ${\rm dim}\,f(W)= l$. Then

(i) If  ${\rm dim}\,f(X_{n-k})<n-k$, one has \begin{eqnarray}\label{eqvanspecial} c_{2k}(X)\alpha^{n-k}=0\,\,{\rm in}\,\,H^*(X,\mathbb{C})
\end{eqnarray}
 for any class $\alpha\in H^2(X,\mathbb{C})$ such that $q(\alpha)=0$.

(ii) If  $X_{n-k} $ is not empty and any component $W$ of $X_{n-k}$ satisfies $${\rm dim}\,W=2n-2k, \,\,{\rm dim}\,f(W)=n-k,$$  the class
 \begin{eqnarray}\label{eqvanspecialsup0} c_{2k}(X)L^{n-k}\in {\rm CH}(X)
\end{eqnarray}
is $\mathbb{Q}$-effective and nonzero.

(iii)  Under the same assumptions as in (ii) for $k=n$, $2n={\rm dim}\,X$, one has $\chi_{\rm top}(X)\geq 0$ and
$ \chi_{\rm top}(X)= 0$ if and only if the locus $X_0$ where $f$ has rank $0$  is empty.
\end{theo}
\begin{rema}{\rm  The vanishing (\ref{eqvanspecial}) is a topological  property which is different from the cohomological  vanishing relations (\ref{relverbchern}) described in the introduction and is not expected to hold in general. In fact, using Theorem \ref{propnewauckland}(ii), one sees that  it does not hold for the known hyper-K\"{a}hler manifolds.}
\end{rema}
\begin{rema} {\rm  Assume there is a commutative group scheme $\mathcal{G}$ over $B$,  with Lie algebra $\Omega_B$, acting on $X$. Then according to \cite[Proposition 5.17]{arinkinfedorov}, the abelian part of $\mathcal{G}_b$ has for dimension  the infimum of the ranks of $f_{x*}$ for $x\in X_b$. In particular the abelian part is nontrivial away from $f(X_0)$ and is nontrivial everywhere if $X_0$ is empty. This reproves Theorem \ref{propnewauckland}(iii) in this case by the classical  argument (see \cite{beauville}). }
\end{rema}

We next  discuss  a refinement of the  condition appearing in Question \ref{questionnewauckland}. We denote by $\pi:\mathcal{G}_k\rightarrow B$  the relative Grassmannian $G(k,T_{B})$ of $k$-dimensional subspaces of $T_{B}$.  Consider the following condition for given $k$.

($*_k$) {\it For any   irreducible  component $W$ of the set
$X_{k}\subset X$, the rational map
\begin{eqnarray}\label{eqpsi}\psi_W: W\dashrightarrow \mathcal{G}_k,\\ \nonumber
 x\mapsto {\rm Im}\,f_{*,x}\subset T_{B,f(x)},\end{eqnarray}
 has image of dimension  $\leq k$.}
 \begin{lemm} \label{lepour*} Condition ($*_k$) holds for a given  component $W$ of $X_k$  if either ${\rm dim}\,W \leq 2k$, or ${\rm dim}\,f(W)=k$.  In particular, it holds if  Question \ref{questionnewauckland} has an affirmative answer for $W$.
 \end{lemm}
\begin{proof} Let $w\in W$. Recall from the proof of Lemma \ref{lemmanewauckland}  that there is a free  local action of  an analytic  germ of commutative group of dimension $k$ on a neighborhood of $w$ in $X$, which is  vertical, that is, preserves $f$. This group action preserves $W$  and the map $\psi_W$ is constant along the orbits of this group acting on $W$, hence if ${\rm dim}\,W \leq 2k$, we have ${\rm dim}\,\psi_W(W)\leq k$.

In the other case, where the generic rank of $f_{\mid W}$ is $k$, the image of $f_*: T_{X,w}\rightarrow T_{B,f(w)}$, for  a general  point $w\in W$,  is equal to the image of  $(f_{\mid W})_*: T_{W,w}\rightarrow T_{B,f(w)}$, that is, to the tangent space of $f(W)$ at the point $f(w)$. In this case, $\psi_W$ is constant along the fibers of $f_{\mid W}$, hence
${\rm dim}\,\psi_W(W)={\rm dim}\,f(W)=k$.
\end{proof}

We conclude this section with a result that will be used in the proof of Theorem \ref{theonewplustechass}. We denote by $$X_{\mathcal{G}_k}=X\times_{B}\mathcal{G}_{k},$$
$$f_\mathcal{G}: X_{\mathcal{G}_k}\rightarrow \mathcal{G}_k,\,\,\pi_X: X_{\mathcal{G}_k}\rightarrow X$$ the fibred product of $X$ and $\mathcal{G}_k$ and its two projections.

\begin{lemm} \label{leimplications} Let $i\leq n$ be an integer.  Assume Condition ($*_k$) of (\ref{eqpsi}) holds for any $k\geq n-i+1$.   Then the  locus $Z\subset  X_{\mathcal{G}_{i-1}}$ of pairs $(x,V),\,x\in X,\,V\subset T_{B,f(x)},\,{\rm dim}\,V=i-1$, such that $f^*:V^{\perp}\rightarrow \Omega_{X,x}$ is not injective, is mapped by $f_\mathcal{G}$ to a  closed algebraic subset  of codimension $\geq i$ in $\mathcal{G}_{i-1}$.
\end{lemm}
\begin{proof}  This is proved   by a dimension count.
The locus $Z$  introduced above is the  union over all $k$ and irreducible components $W$ of $X_k$ of the loci $Z_W:=Z\cap \pi_X^{-1}(W)$. As ${\rm dim}\,f(W)\leq k$ for $W\subset X_k$, the loci  $ \pi_X^{-1}(W)$ are  of codimension  $\geq i$ for $n-k\geq i$, so we can assume that $k\geq n-i+1$.
At a point $x\in  \pi_X^{-1}(W)$, the map $f^*_x:\Omega_{B,f(x)}\rightarrow \Omega_{X,x}$ has rank exactly $k$, and for a $n-i+1$-dimensional subspace $V^{\perp}\subset \Omega_{B,f(x)}$, the condition that $f^*_x:V^{\perp}\rightarrow \Omega_{X,x}$ is not injective says that $V^{\perp}\cap {\rm Ker}\,f^*_x\not=\{0\}$. This imposes $k-n+i$ Schubert conditions on $V$, that are determined by the point ${\rm Im}\,f_{*,x}\in \mathcal{G}_k$. It follows that the codimension of the image of  $Z_W$ in  $\mathcal{G}_{i-1}$ is at least $$k-n+i+{\rm codim}\,f(W)-{\rm dim}\,\psi_W(W_b),$$ where $b\in f(W)$ is a general point and $W_b:=f^{-1}(b)\cap W$. As $${\rm codim}\,f(W)=n-{\rm dim}\,f(W),\,\,\,\,{\rm dim}\,{\rm Im}\,\psi_W={\rm dim}\,f(W)+{\rm dim}\,\psi_W(W_b),$$  we conclude that  the codimension of the image of  $Z_W$ in  $\mathcal{G}_{i-1}$ is at least  $i+k-{\rm dim}\,({\rm Im}\,\psi_W)$.  By assumption,  ${\rm dim}\,({\rm Im}\,\psi_W)\leq k$ for all $k\geq n-i+1$ and $W$, so we get that the codimension of the image of  $Z_W$ in  $\mathcal{G}_{i-1}$ is at least $i$ for all $W$.
\end{proof}

\subsection{Support for Chern classes of Lagrangian fibered varieties \label{secsupport}}

We prove in  this section  the following result.
\begin{theo}\label{theonewplustechassavec*} Let $f: X\rightarrow B$ be a Lagrangian fibration of a projective hyper-K\"{a}hler manifold and $i$ be a positive integer. Assume that $B$ is smooth in codimension $i-1$ and that the condition ($*_k$) of the previous section (see (\ref{eqpsi}))  is satisfied for $k\geq n-i+1 $.
Then  for any $j\geq i$, $c_{2j}(X)$ vanishes in ${\rm CH}(X\setminus f^{-1}(B^i))$ for some codimension $i$ closed algebraic subset $B^i$ of $B$.
\end{theo}
We first  establish some  consequences.
\begin{coro} \label{coroideux}  Let $f: X\rightarrow B$ be a Lagrangian fibration of a projective hyper-K\"{a}hler manifold, where we assume  $B$ normal. Then $c_2(X)$ vanishes in  ${\rm CH}(X\setminus f^{-1}(B^1))$  and $c_{2j}(X) $ vanishes in  ${\rm CH}(X\setminus f^{-1}(B^2))$ for any $j\geq 2$, where $B^i\subset B$, $i=1,\,2$, is a closed algebraic subset of codimension $\geq i$.
\end{coro}
\begin{proof} As $B$ is normal, its singular locus has codimension $\geq 2$, so in both cases,  in order to apply Theorem \ref{theonewplustechassavec*}, we only have to study the condition  ($*_k$)  for $k\geq n-i+1$ for $i=1,\,2$.  The assumption ($*_k$) obviously holds for $k=n$, implying the first statement. For the second statement, where we have $i=2$,  we only need to study the sets $X_{k}$ for $k=n-1$. Let $W$ be an irreducible component of $X_{n-1}$. Then the image of $W$ under $f$ has dimension $\leq n-1$, since $f_{\mid W_{\rm reg}}$ has rank $\leq n-1$ everywhere. If ${\rm dim}\,f(W)=n-1$, then ($*_{n-1}$) is satisfied by $W$ by Lemma \ref{lepour*}. There might be other components $W$ for which  ${\rm dim}\,f(W)\leq n-2$, but  we can restrict the Lagrangian fibration over the open set $B\setminus \cup_{W\subset X_{n-1},\,{\rm dim}\,f(W)\leq n-2} f(W)$ and apply Theorem \ref{theonewplustechassavec*} over this open set.
\end{proof}
\begin{proof}[Proof of Theorem \ref{theonewplustechass}] By Lemma \ref{lepour*}, the assumptions of Theorem \ref{theonewplustechass} imply that Condition ($*_k$) holds for $k\geq n-i+1$. Thus Theorem \ref{theonewplustechassavec*} implies Theorem \ref{theonewplustechass}.
\end{proof}
For the proof of Theorem \ref{theonewplustechassavec*}, we will use the following   general lemma about base change invariance of vanishing of Chern classes.
\begin{lemm} \label{lemmachaint} Let $f: Y\rightarrow M$ be a proper morphism with $M$ smooth and let $\pi: N\rightarrow M$ be a smooth  proper morphism. Denote by $Y_N $ the fibered product $Y\times_MN$ and by $$f_N:Y_N\rightarrow N,\,\pi_Y:Y_N\rightarrow Y$$ the two projections. Let
 $E$ be a vector bundle on $Y$ and let $l$ be a positive integer. Then the Chern class $c_l(E)$ vanishes in ${\rm CH}^l(Y\setminus f^{-1}(Z))$ for some codimension $k$ closed algebraic subset $Z$  of $M$  if and only if the Chern class $c_l(\pi_Y^*E)$ vanishes in ${\rm CH}^l(Y_N\setminus f_N^{-1}(Z'))$ for some codimension $k$ closed algebraic subset $Z'$ of $N$.
\end{lemm}
\begin{proof} If $c_l(E)$ vanishes in ${\rm CH}^l(Y\setminus f^{-1}(Z))$ for some codimension $k$ closed algebraic subset $Z$  of $M$, $c_l(\pi_Y^* E)$ vanishes in ${\rm CH}^l(Y_N\setminus \pi_Y^{-1}(f^{-1}(Z)))$, and, as $f\circ \pi_Y=\pi\circ f_N$, we have  $\pi_Y^{-1}(f^{-1}(Z))= f_N^{-1}(Z')$, where $Z'=\pi^{-1}(Z)$ is of codimension $k$ by smoothness of $\pi$.

Conversely, let $H$ be a very ample line bundle on $N$ and let $d$ be the relative dimension of $\pi$. Assume that the Chern class $c_l(\pi_Y^*E)$ vanishes in ${\rm CH}^l(Y_N\setminus f_N^{-1}(Z'))$ for some codimension $k$ closed algebraic subset of $N$.  Then considering the complete intersection  of  $d$ general members of $|H|$, we get a  smooth closed algebraic subset $N_1\subset N$, that maps in a generically finite way to $M$ via $\pi_1:=\pi_{\mid N_1}$, and a codimension $k$  closed algebraic subset $Z'_1\subset N_1$, whose image $Z:=\pi_1(Z'_1)\subset M$ has codimension $k$. As $c_l(\pi_Y^*E)$ vanishes in ${\rm CH}^l(Y_N\setminus f_N^{-1}(Z'))$,
we get that $c_l(\pi_{1,Y}^*E)$ vanishes in ${\rm CH}^l(Y_{N_1}\setminus f_{N_1}^{-1}(Z'_1))$, where $Y_{N_1}:= Y\times_MN_1$, with projections $\pi_{1,Y}:Y_{N_1}\rightarrow Y$, $f_{N_1}:Y_{N_1}\rightarrow N_1$. As $Z=\pi_1(Z'_1)$, we find  that
$\pi_{1,Y*}(c_l(\pi_{1,Y}^*E))$ vanishes on $Y\setminus f^{-1}(Z)$ and this concludes the proof since $\pi_{1,Y*}(c_l(\pi_{1,Y}^*E))=({\rm deg}\,\pi_1) c_l(E)$.
\end{proof}

Let now $f: X\rightarrow B$ be a Lagrangian fibration, where $X$ is projective hyper-K\"{a}hler of dimension $2n$.  Let $B':=B\setminus B_{\rm sing}$ and  $X':=f^{-1}(B')$. The variety $\mathcal{G}_{i-1}$ introduced in the previous section  is smooth of dimension $n+(n-i+1)(i-1)$ and fibered into Grassmannians over $B'$ via a smooth proper morphism  that we denote by $\pi$. It carries a tautological
subbundle $\mathcal{S}$ of rank ${n-i+1}$ of  the bundle $\pi^*\Omega_{B'}$, with fiber $V^{\perp}\subset \Omega_{B',b}$ at a point $$(b,[V]), V\subset T_{B',b},\,{\rm dim}\,V=i-1$$ of $\mathcal{G}_{i-1}$.
We denote as before by $f_\mathcal{G}:X'_\mathcal{G}\rightarrow \mathcal{G}_{i-1}$, $\pi_X:X'_\mathcal{G}\rightarrow X'$ the two projections.
We have on $X'$ the  morphism of vector bundles
$$\phi:=f^*:f^*\Omega_{B'}\rightarrow \Omega_{X'}$$ and thus, by pull-back to $X'_\mathcal{G}$ and restriction to $\mathcal{S}$, we get a morphism
\begin{eqnarray}\label{eqphiS} \phi_{\mathcal{S}}:\mathcal{S}\rightarrow \pi_X^*\Omega_{X'}.
\end{eqnarray}

We now conclude the proof of  Theorem \ref{theonewplustechassavec*}.
\begin{proof}[Proof of Theorem \ref{theonewplustechassavec*}]  By  Lemma \ref{leimplications},  assumption ($*_k$) for $k\geq n-i+1$ implies that   the morphism $\phi_{\mathcal{S}}$ is injective on $X_\mathcal{G}\setminus f_\mathcal{G}^{-1}(Z_i)$ for some closed algebraic subset $Z_i$ of $\mathcal{G}_{i-1}$ of codimension $\geq i$. As $X$ is hyper-K\"ahler, the vector bundle $\Omega_X$  carries an everywhere nondegenerate skew-symmetric pairing $\langle\,,\,\rangle$. As $f$ is Lagrangian, the morphism $f^*:f^*\Omega_{B'}\rightarrow \Omega_{X'}$ is Lagrangian and it follows that the image of the morphism $\phi_{\mathcal{S}}$, which  by definition of $Z_i$ is a subbundle of rank $n-i+1$ on $X'_\mathcal{G}\setminus f_\mathcal{G}^{-1}(Z_i)$, is totally isotropic at any point of $X'_\mathcal{G}\setminus f_\mathcal{G}^{-1}(Z_i)$. The subbundle $({\rm Im}\,\phi_{\mathcal{S}})^{\perp}$ of $\pi_X^*\Omega_{X'}$ is thus of rank $n+i-1$ on $X'_\mathcal{G}\setminus f_\mathcal{G}^{-1}(Z_i)$  and contains ${\rm Im}\,\phi_{\mathcal{S}}$. Furthermore, the quotient $\Omega_X/({\rm Im}\,\phi_{\mathcal{S}})^{\perp}$ is isomorphic to $f_\mathcal{G}^*\mathcal{S}^*$ via $\langle\,,\,\rangle$. Let
 $$\mathcal{E}:= ({\rm Im}\,\phi_{\mathcal{S}})^{\perp}/{\rm Im}\,\phi_{\mathcal{S}}.$$
 This is a vector bundle of rank $2i-2$ on $X'_\mathcal{G}\setminus f_\mathcal{G}^{-1}(Z_i)$, hence we have
 \begin{eqnarray}\label{eqannurankbis} c_{2i}(\mathcal{E})=0\,\,{\rm in}\,\,{\rm CH}(X'_\mathcal{G}\setminus f_\mathcal{G}^{-1}(Z_i)).\end{eqnarray}
We use the exact sequences
$$0\rightarrow f_\mathcal{G}^*\mathcal{S}\rightarrow ({\rm Im}\,\phi_{\mathcal{S}})^{\perp}\rightarrow \mathcal{E}\rightarrow 0,$$
$$0\rightarrow ({\rm Im}\,\phi_{\mathcal{S}})^{\perp}\rightarrow \pi_X^*\Omega_{X'}\rightarrow f_\mathcal{G}^*\mathcal{S}^*\rightarrow 0$$
explained above and the Whitney formula, which gives equalities in  ${\rm CH}(X'_\mathcal{G}\setminus f_\mathcal{G}^{-1}(Z_i))$
\begin{eqnarray}\label{eqW1bis} \pi_X^*c(\Omega_{X'})=f_\mathcal{G}^*c(\mathcal{S}^*) c(({\rm Im}\,\phi_{\mathcal{S}})^{\perp})\\
\nonumber
=f_\mathcal{G}^*c(\mathcal{S}^*)f_\mathcal{G}^*c(\mathcal{S}) c(\mathcal{E}).
\end{eqnarray}
We also get by inverting the total Chern classes in (\ref{eqW1bis})
\begin{eqnarray}\label{eqW2bis} c(\mathcal{E})=\pi_X^*c(\Omega_{X'})f_\mathcal{G}^*c(\mathcal{S}^*)^{-1}f_\mathcal{G}^*c(\mathcal{S})^{-1}\,\,{\rm in}\,\,{\rm  CH}(X_\mathcal{G}^0\setminus f_\mathcal{G}^{-1}(Z)).
\end{eqnarray}
We now conclude the proof by induction on $i$. The induction assumption tells us that for $j<i$, the Chern class $c_{2j}(\Omega_X)$ vanishes in  ${\rm CH}(X'\setminus f^{-1}(B^j))$ for some codimension $j$ closed algebraic subset of $B'$. By pull-back, the Chern class $c_{2j}(\pi_X^*\Omega_{X'})$ vanishes in  ${\rm CH}(X'_\mathcal{G}\setminus f_G^{-1}(Z_j))$ for some codimension $j$ closed algebraic subset $Z_j$ of $G_{n-i+1}$, for $j<i$.
It follows from (\ref{eqW2bis}) that, for $j<i$, the Chern class $c_{2j}(\mathcal{E})$ vanishes in  ${\rm CH}(X'_\mathcal{G}\setminus (f_\mathcal{G}^{-1}(Z_i\cup Z'_j))$ for some codimension $j$ closed algebraic subset $Z'_j$ of $\mathcal{G}_{i-1}$ (note that $\mathcal{E}$, being self-dual, has trivial odd Chern classes). It follows that any class of the form
$f_\mathcal{G}^* \alpha\cdot c_{2j}(\mathcal{E})$, where $\alpha\in {\rm CH}^{2i-2j}(\mathcal{G}_{i-1})$, vanishes in  ${\rm CH}^{2i}(X'_\mathcal{G}\setminus (f_\mathcal{G}^{-1}(Z_i\cup Z'_i))$ for some codimension $i$ closed algebraic subset $Z'_i$ of $\mathcal{G}_{i-1}$. By (\ref{eqannurankbis}), this is also true for $c_{2i}(\mathcal{E})$.  Using (\ref{eqW1})  and expanding the product in degree $2i$, we conclude  that $c_{2i}(\pi_X^*\Omega_X)$ vanishes in  ${\rm CH}(X'_\mathcal{G}\setminus f_\mathcal{G}^{-1}(Z_i\cup Z'_i)))$, where $Z_i$ and $Z'_i$ have codimension $\geq i$ in $\mathcal{G}_{i-1}$. By Lemma \ref{lemmachaint}, we conclude that
$c_{2i}(\Omega_{X'})$ vanishes in  ${\rm CH}(X'\setminus f^{-1}(B^i))$ for some codimension $i$ closed algebraic subset of $B'$. As  $B\setminus B'$ has codimension $\geq i$ in $B$ by assumption, $c_{2i}(\Omega_{X})$ also vanishes in  ${\rm CH}(X\setminus f^{-1}(B^i))$ for some codimension $i$ closed algebraic subset of $B$.
\end{proof}
\subsection{Proof of Theorem \ref{propnewauckland}}

\begin{proof}[Proof of Theorem  \ref{propnewauckland}](i) First of all, we observe that the mapping class group  acts   on the cohomology algebra $H^*(X,\mathbb{Q})$ preserving the Chern classes $c_i(X)\in H^{2i}(X,\mathbb{Q})$. Indeed, the Chern classes of a hyper-Kähler manifold are in fact determined by  its Pontryagin classes, that are topological invariants.  Furthermore, it is known by results of Verbitsky that this action restricted to  $H^2(X,\mathbb{Z})$ is that of  a finite index subgroup of $O(q)$.    The orbit under  this action of any nonzero class $l$ with $q(l)=0$ is Zariski dense in the quadric $\{q=0\}\subset H^2(X,\mathbb{C})$.  It thus suffices to show that, under the  assumptions of  Theorem \ref{propnewauckland}(i), we have
\begin{eqnarray}\label{eqvancohauckalnd}  c_{2k}(X) l^{n-k}=0\,\,{\rm in}\,\,H^{2*}(X,\mathbb{Q}),
\end{eqnarray}
since it will imply the same result with $l$ replaced by  any $\alpha\in H^2(X,\mathbb{C})$ with $q(\alpha)=0$.
A fortiori, it suffices to show that
\begin{eqnarray}\label{eqvancohauckalndchow}  c_{2k}(X) L^{n-k}=0\,\,{\rm in}\,\,{\rm CH}^*(X).
\end{eqnarray}

We choose a general complete intersection $B^{n-k}\subset B$ of $n-k$ ample hypersurfaces, whose inverse images in $X$ belong to $|dL|$ for some $d$, and denote by $X^{n-k}\subset X$ its inverse image in $X$, so that the class of $X^{n-k}$ in ${\rm CH}(X)$ is a nonzero multiple of $L^{n-k}$. As
$X^{n-k}$ is a smooth complete intersection, we have the cotangent bundle  exact sequence
\begin{eqnarray}\label{eqpremapauck}0\rightarrow \mathcal{O}_{X^{n-k}}(-dL)^{n-k}\rightarrow \Omega_{X\mid X^{n-k}}\rightarrow \Omega_{X^{n-k}}\rightarrow 0.\end{eqnarray}
As we already used several times, the image of the map $\mathcal{O}(-dL)^{n-k}\rightarrow \Omega_{X\mid X^{n-k}}$ is totally isotropic and there is thus
a dual surjective morphism
\begin{eqnarray}\label{eqsecapauck} \Omega_{X^{n-k}}\rightarrow  \mathcal{O}_{X^{n-k}}(dL)^{n-k}\rightarrow 0\end{eqnarray}
whose kernel is a rank $2k$ vector bundle on $X^{n-k}$, which will be denoted by $\mathcal{E}$.
Using (\ref{eqpremapauck}) and (\ref{eqsecapauck}), we get formulas in ${\rm CH}(X^{n-k})$
\begin{eqnarray} \label{eqthirdcapauck}c_{2k}(\Omega_{X\mid  X^{n-k}})=c_{2k}(\mathcal{E})+\sum_{l>0}\alpha_l  L^{2l} c_{2k-2l}(\mathcal{E}),\\ \nonumber
c_{2i}(\mathcal{E})=c_{2i}(\Omega_{X\mid  X^{n-k}})+ \sum_{l>0}\beta_l  L^{2l}  c_{2i-2l}(\Omega_{X\mid  X^{n-k}}),
\end{eqnarray}

As we know by Theorem \ref{theonew} that $L^{n-k+2l} c_{2k-2l}(X)=0$ for $l>0$, we deduce from  (\ref{eqthirdcapauck}) that the desired vanishing (\ref{eqvancohauckalndchow}) is a consequence of the vanishing

\begin{eqnarray} \label{eqfourcapauck} c_{2k}(\mathcal{E})=0\,\,{\rm in}\,\,{\rm CH}(X^{n-k})
\end{eqnarray}
that we prove now, using  the assumptions of Theorem \ref{propnewauckland}(i).
We denote by $f': X^{n-k}\rightarrow B^{n-k}$ the restriction of $f$. Using the fact that $f$ is a Lagrangian fibration, the differential
$${f'}^*:  {f'}^* \Omega_{B^{n-k}}\rightarrow    \Omega_{X^{n-k}}$$
induces a morphism
\begin{eqnarray}\label{eqphiprimeauck} \phi':  {f'}^* \Omega_{B^{n-k}}\rightarrow \mathcal{E}.
\end{eqnarray}
 Note that $\mathcal{E}$ has an induced nondegenerate skew-symmetric form and that $\phi'$ is a Lagrangian morphism of vector bundles in the sense of Definition \ref{defimorlag}.
Let  $$\pi: P:=\mathbb{P}(\Omega_{B^{n-k}})\rightarrow B^{n-k},\,\,X^{n-k}_P:=P\times_{B^{n-k}} X^{n-k}$$ and denote by $$f'_P:X^{n-k}_P\rightarrow P,\,\,\pi_X:X^{n-k}_P\rightarrow X$$ the two projections. Let
$\mathcal{S}\subset \pi^*\Omega_{B^{n-k}}$ be the rank $1$ tautological subbundle. The morphism $\phi'$ induces
a morphism $\phi'':{f'_P}^*\mathcal{S}\rightarrow  \pi_X^*\mathcal{E}$, or equivalently  a section
\begin{eqnarray}\label{eqalphahauck} \alpha\in H^0(X_P,  {f'_P}^*\mathcal{S}^{-1}\otimes \pi_X^*\mathcal{E}).\end{eqnarray}
The section $\alpha$ is not transverse and we are going to describe its vanishing locus $Z(\alpha)$ below. After applying an excess formula \`{a} la Fulton \cite[Section 6.3]{fulton}, the corrected vanishing locus  will have  class
\begin{eqnarray}\label{eqZalpahauck}Z(\alpha)_{\rm vir}=c_{2k}({f'_P}^*\mathcal{S}^{-1}\otimes \pi_X^*\mathcal{E})=\sum_{0\leq 2i\leq 2k} {f'_P}^*(c_{1}(\mathcal{S}^{-1})^{2i})\pi_X^*c_{2k-2i}(\mathcal{E}).\end{eqnarray}
Recall that $\pi_X: X_P\rightarrow X$ is the  projectivized bundle $\mathbb{P}({f'}^*\Omega_{B^{n-k}})$, polarized by ${f'_P}^*\mathcal{S}^{-1}$. We note here for future use  that, as we assumed that $B$ is smooth, it is isomorphic to $\mathbb{P}^n$ by \cite{hwang}, hence  we can assume that $B^{n-k}$ is a $\mathbb{P}^{k}$, and  $\mathcal{S}^{-1}$ is very ample on $P$, with space of global sections $H^0(\mathbb{P}^{k},T_{\mathbb{P}^{k}})$. We  get from (\ref{eqZalpahauck})

\begin{eqnarray}\label{eqZalpahauckpousse}  \pi_{X*}(Z(\alpha)_{\rm vir}c_1({f'_P}^*\mathcal{S}^{-1})^{k-1})=
c_{2k}(\mathcal{E})+\sum_{i>0} {f'}^*s_{2i}(\Omega_{B^{n-k}})c_{2k-2i}(\mathcal{E}),
\end{eqnarray}
where the $s_{2i}$'s denote the Segre classes.
One proves as in Theorem \ref{theonew} that $${f'}^*s_{2i}(\Omega_{B^{n-k}})c_{2k-2i}(\mathcal{E})=0\,\,{\rm in}\,\,{\rm CH}^{2k}(X^{n-k})$$ for $i>0$,  so we get in fact the equality
  \begin{eqnarray}\label{eqZalpahauckpoussecorrect} \pi_{X*}(Z(\alpha)_{\rm vir} {f'_P}^*(c_1(\mathcal{S}^{-1})^{k-1}))=
c_{2k}(\mathcal{E})   \,\,{\rm in}\,\, {\rm CH}^{2k}(X^{n-k}).
\end{eqnarray}
We now describe the locus $Z(\alpha)$. We observe that $X^{n-k}$ is stratified by strata
$$X^{n-k}_l:=X^{n-k}\cap X_{n-k+l}$$ where the rank of $f_*$ is $n-k+l$, or equivalently the rank of $f'_*$ is $l$. For $l=0$, as we assumed that  ${\rm dim}\,f(X_{n-k})<k$, $X^{n-k}_0$ is  empty. We will now analyze the contributions of the strata $X_l^{n-k}$ for $l>0$ and show that they all vanish. For $l>0$, our assumption is that $X_{l+n-k}$ has dimension $2l+2n-2k$ and its image $B_l^{n-k}$ in $B$ has dimension $l+n-k$, so we have
${\rm dim}\,X^{n-k}_l=2l+n-k$ and its image in $B$ has dimension $l$.
We now observe that, denoting $X^{n-k}_{P,l}$ the inverse image $\pi_X^{-1}(X_l^{n-k})$,  the section
$\alpha$ of  (\ref{eqalphahauck}) vanishes along the projective subbundle
$$\mathbb{P}(\mathcal{K}_l)\subset \mathbb{P}(({f'}^*\Omega_{B^{n-k}})_{\mid X^{n-k}_l})= X^{n-k}_{P,l},$$ where the vector bundle
$\mathcal{K}_l$ on $X^{n-k}_l$ is the kernel of ${f'}^*$, hence  has rank $k-l$. We thus conclude that  the intersection $Z(\alpha)_l$ of  the locus $Z(\alpha)$  with the   stratum $X_{P,l}^{n-k}$ has dimension
\begin{eqnarray}\label{eqdim8mai}{\rm dim}\,Z(\alpha)_l={\rm dim}\,X^{n-k}_l+ {\rm rk}\,\mathcal{K}_l-1=2l+n-k+k-l-1=n+l-1.
\end{eqnarray}  Note that the expected dimension of $Z(\alpha)$ is $${\rm dim}\,X^{n-k}_P-2k=2n-(n-k)+k-1-2k=n-1.$$

The key point is now the following: generically along $X_l^{n-k}$, the rank $k-l$ vector bundle $\mathcal{K}_l$ comes from the rank $k-l$  subbundle on $B^{n-k}_l$ with fiber ${\rm Ker}\,(\Omega_{B^{n-k}\mid B_l^{n-k}}\rightarrow \Omega_{B^{n-k}_l})$. This follows indeed from our assumptions that $${\rm dim}\,X_{l+n-k}=2(n+l-k),\,{\rm dim}\,B_{l+n-k}=n+l-k,$$ which implies that the generic rank of $f_{\mid X_{n-k+l}}$ is $n-k+l$ (cf. Question \ref{questionnewauckland}), hence the generic rank of $f'_{\mid X^{n-k}_l}$ is $l$.
It follows from this fact   that the image  $f'_P(Z(\alpha)\cap X_l^{n-k})$ in $P$ has dimension at most
$${\rm dim}\,B^{n-k}_l+k-l-1=l+k-l-1=k-1.$$
 When we intersect  $Z(\alpha)$ with $k-1$ hypersurfaces in  $|{f'_P}^*\mathcal{S}^{-1}|$ coming from $P$, we thus get a subvariety  $Z'(\alpha)$ which is supported over finitely many points in $P$, and its image
in  $X^{n-k}$ under $\pi_X$ has  finitely many irreducible components, all contained in fibers of $f'$. More precisely, the   fibers of $Z'(\alpha)$  supported over $B_l$ are supported over generic points of $B_l^{n-k}$, contained in  $X_l^{n-k}$ and  of dimension $n-k+l$ by (\ref{eqdim8mai}).
Finally, recall that, by formula (\ref{eqZalpahauckpoussecorrect}),  we have  to compute
$ \pi_{X*}(Z(\alpha)_{\rm vir} {f'_P}^*(c_1(\mathcal{S}^{-1})^{k-1}))$ and show that it is $0$ in ${\rm CH}(X^{n-k})$. Here the cycle $Z(\alpha)_{\rm vir}$ is supported on $Z(\alpha)$ and computed by applied Fulton's refined intersection formula \cite{fulton}. Now we observe that we can replace the term  $Z(\alpha)_{\rm vir} {f'_P}^*(c_1(\mathcal{S}^{-1})^{k-1})$ appearing in this formula by the virtual class of  $ Z'(\alpha) $. As we described $Z'(\alpha)$ as a disjoint union of components  $Z'(\alpha)_{l, b_i}$ of dimension $n-k+l$ contained in fibers $X^{n-k}_{P,l,b_i}$, it suffices to show that the  contribution of each
$Z'(\alpha)_{l, b_i}$ to $ Z'(\alpha)_{\rm vir} $ is $0$ for $l>0$. This follows from the following
\begin{claim}  The excess bundle for the section $\alpha$ is trivial of rank $l$ along the  fiber $X^{n-k}_{P,b}\cap Z(\alpha)$, for  a general point $b\in B^{n-k}_l$.
\end{claim}
\begin{proof} Let $b\in B^{n-k}_l$ be  a general point. Choose $l$ functions $g_1,\ldots, g_l$  on $B^{n-k}$ whose differentials at $b$  are independent on $B^{n-k}_l$. Then their pull-backs to $X^{n-k}$ have independent differentials along  $X_{l,b}^{n-k}$, hence in a Zariski open neighborhood $U_b^{n-k}$ of $X_{l,b}^{n-k}$ in $X^{n-k}$, since $b\in B^{n-k}_l$ is a general point so the morphism $X^{n-k}_l\rightarrow B^{n-k}_l$ has to be of rank $l$ at any point over $b$. Let $\mathcal{F}$ be the trivial rank $l$ vector bundle on $U_b^{n-k}$ generated by these differentials. We have   a natural morphism $\mathcal{F}\rightarrow \mathcal{E}$ induced by the inclusion $\mathcal{F}\rightarrow \Omega_{U^{n-k}_b}$, using the fact that $f$ is Lagrangian. By duality, this  induces a quotient map
$$q: \mathcal{E}\rightarrow \mathcal{F}^*\rightarrow 0$$
on $U_b^{n-k}$.
The morphism
$${f'}^*: {f'}^*\Omega_{B^{n-k}}\rightarrow \mathcal{E},$$
being Lagrangian,
takes value in ${\rm Ker}\,q$. This proves that the excess bundle identifies to $\mathcal{F}^*$ along $X^{n-k}_{l,b}$. This proves the result since $\mathcal{F}$ is trivial on $U^{n-k}_b$, hence on $X^{n-k}_b$.
\end{proof}
 The claim  concludes the proof since, as we already mentioned, by our  assumption  that ${\rm dim}\,B_{n-k}<n-k$, $B^{n-k}_0$ is empty so  there is no contribution from the stratum where $l=0$.
\end{proof}
\begin{proof}[Proof of Theorem \ref{propnewauckland}(ii) and (iii)] Statement (iii) is a particular case of (ii). For the proof of (ii), we just follow the analysis made previously, except that now,  the stratum $X^{n-k}_0$ is not empty and of the right dimension by (\ref{eqdim8mai}). In particular $B^{n-k}_0$ is a  non-empty set of points and  $X^{n-k}_0$ has dimension $n-k$, and there is no excess in the  contribution of the  stratum $X^{n-k}_0$  to  the cycle
$p_{X*} (c_1(\mathcal{S}^{-1})^{k-1} Z(\alpha)_{\rm vir})$, and this stratum contributes via an effective non-empty cycle of dimension $n-k$. The  contributions of the other strata are zero as explained in the previous proof. This shows that the cycle $c_{2k}(X) L^{n-k}$ is effective and nontrivial.
\end{proof}
\section{Proof of Theorems \ref{theoLSV} and \ref{theoriess}}
\subsection{The cycle   $c_2^2L^{n-1}$ \label{secc2square}}
We study in this section the cycle  $L^{n-1}c_2(X)^2$, where $L$ is a Lagrangian line bundle on a projective hyper-K\"{a}hler $2n$-fold $X$, whose  vanishing is predicted by Conjecture \ref{conjvan}.
We use some of the constructions and notations introduced in the previous sections. We denote by $C=B^{n-1}\subset B$ the (smooth) complete intersection of $n-1$ ample general hypersurfaces in a very ample  linear system $|H|$ on $B$ with $f^*H=L^{\otimes d}$ and  $X_C=X^{n-1}\subset X$ its inverse image in $X$ with restricted morphism
$f':X_C\rightarrow C$.  We have on $X_C$ the rank $2$ vector bundle $\mathcal{E}$   with trivial determinant, which is the kernel  of the natural surjective  morphism  $\Omega_{X_C}\rightarrow \mathcal{O}_{X_C}(dL)^{n-1}$ of (\ref{eqsecapauck}), and has an induced symplectic structure (see also  (\ref{eqcalEnew})).
We only  study  the case where there are no nonreduced fibers in codimension $1$. The vanishing locus of the  morphism
${f'}^*: {f'}^*\Omega_C\rightarrow \Omega_{X_C}$
thus has  codimension $\geq 2$ in $X_C$. Furthermore, as we saw already, this morphism takes values in $\mathcal{E}$ and its vanishing locus $Z$ thus represents the class $c_2(\mathcal{E}(-{f'}^*K_C))$. Note that $Z$ is supported on fibers of $f'$, on which ${f'}^*K_C$ is trivial. Arguing as in the proof of Theorem \ref{propnewauckland}, that is, comparing $j_*c_2(\mathcal{E})$, where $j$ is the inclusion of $X_C$ in $X$,  and $L^{n-1}c_2(X)$,
we get
\begin{lemm} \label{proppourc22}  The class $L^{n-1}c_2(X)^2\in {\rm CH}^{n+3}(X)$ is proportional to  the class $j_{Z*}(c_2(\mathcal{E})_{\mid Z})$, where $Z$ is as above  the critical locus of $f'$ and $j_Z:Z\rightarrow X$ is the inclusion map.
\end{lemm}

The geometry of $Z$ and $\mathcal{E}$ is very interesting and a particular case of a phenomenon studied in \cite{claudon}, \cite{druel}. Indeed, we can apply Theorem \ref{theopourrankloci} with $k=n-1$ since for $k=n-1$,  the equalities in (1) and (2) hold in codimension $1$ on $B$ once there are no nonreduced fibers in codimension $1$.  We thus  conclude that $Z$ is a finite union of abelian varieties, which are projective leaves of a foliation.
The conormal bundle of a leaf $F$ of a foliation $\mathcal{F}:=\mathcal{E}^{\perp}$ on a manifold $Y$ admits an integrable holomorphic connection, defined as the composition
$$\Omega_Y\supset \mathcal{E}\stackrel{d}{\rightarrow} \mathcal{E}\wedge\Omega_Y\rightarrow \mathcal{E}\otimes \Omega_{F},$$
 so it is a flat bundle. Unfortunately, we cannot say much more about it, since by the suspension construction described in
 \cite[Example 9.1]{druel}, any flat vector bundle $E$ on an abelian variety $A$ is isomorphic to the conormal bundle of a foliation on the total space of $E$, of which $A$ is a leaf. Furthermore, the Chern classes of a flat holomorphic vector bundle with trivial determinant on an abelian variety $A$ can be  nontrivial in ${\rm CH}(A)$. For example, for any line bundle $M\in {\rm Pic}^0(A)$, the vector bundle $M\oplus M^{-1}$ has trivial determinant but its second Chern class $-M^2\in{\rm CH}^2(A)$ is in general nontorsion. In fact, once the dimension of $A$ is at least $2$, the subgroup of ${\rm CH}^2(A)$ generated by these classes is infinite dimensional in the Mumford sense, by Mumford's theorem \cite{mumford}.
In the present situation, we have
\begin{lemm}\label{lepouE} For each connected component $Z_i\subset X_c$ of $Z$, the vector bundle $\mathcal{E}_{\mid Z_i}$ is isomorphic to either

(i) a direct sum $M_i\oplus M_i^{-1}$ for some line bundle $M_i\in{\rm Pic}^0(Z_i)$, or

(ii)  a tensor product $M_i\otimes U$ for some $2$-torsion line bundle $M_i$ on $Z_i$ and rank $2$ vector bundle $U$ which is a an extension of the trivial line bundle by itself.
\end{lemm}
\begin{proof} The vector bundle $\mathcal{E}_{\mid Z_i}$  is  homogeneous by Theorem \ref{theopourrankloci}. Homogeneous rank $2$ vector bundles on abelian varieties are classified in \cite{mukai} and are direct sums of homogeneous (that is topologically trivial) line bundles, which gives the first case, or of the form $M_i\otimes U$ as in (ii). In case (ii), $M_i$ has to be a $2$-torsion line bundle since ${\rm det}\,\mathcal{E}$ is trivial.
\end{proof}
In case (ii), the class $c_2(\mathcal{E}_{\mid Z_i})$ is trivial, and thus the contribution of $Z_i$ to the class $c_2(X)L^{n-1}$ is zero.  Unfortunately, case (i) is the most naturally encountered. We discuss another more classical viewpoint on the line bundle $M_i$ in  the following Proposition  \ref{lepourMi14mai}, showing that there   is no restriction on  the topologically trivial  line bundle $M_i$ in (i). Let $c\in C$ be a critical value of $f'$, let $X_c$ be the fiber $f^{-1}(c)$ and $Z_c\subset X_c$ be the singular locus  of  $X_c$. By \cite{hwangoguiso}, the normalization $\widehat{X}_c$ is smooth and it is a $\mathbb{P}^1$-bundle over an \'{e}tale cover of an  abelian variety $A_c$,  via its Albanese map $$a:\widehat{X}_c\rightarrow A_c.$$ The inverse image of $Z_c$ is a finite union of abelian varieties contained in
$  \widehat{X}_c$ which are multisections of $a$ and \'{e}tale over $A_c$. These facts  follow   again from the existence of the local vertical  free group action of $\mathbb{C}^{n-1}$ on   a neighborhood of $X_c$ in $X_C$, which is transitive on the components of $Z_c$, and  provides automorphisms of $X_c$ which lift to $\widehat{X}_c$.
\begin{prop} \label{lepourMi14mai}  Assume we are in Case (i) of Lemma \ref{lepouE}, and that  $M_i$ is not  torsion,  $X_c$ is irreducible,    $X_c$ has at  most   two branches  along $Z_i$, and the branches  are smooth. Then the inverse image $\widehat{Z}_i:=n^{-1}(Z_i)\subset \widehat{X}_c$ is a  divisor which has  degree $2$ over $Z_i$  and  there is a canonical isomorphism
\begin{eqnarray}\label{eqiso14mai} \mathcal{E}_{\mid Z_i}\cong n_*(\mathcal{O}_{\widehat{Z}_i}(\widehat{Z}_i)).
\end{eqnarray}
Furthermore $\widehat{Z}_i$ is the union of  two disjoint divisors  $\widehat{Z}_{i,1},\,\widehat{Z}_{i,2}$  isomorphic to $A_c$ via  $a: \widehat{X}_c\rightarrow A_c$ and to $Z_i$ via $n$
and, for any ample line bundle $H$ on $X$, we have for some nonzero integer $d$
\begin{eqnarray}\label{eqisopourMi14mai}   dM_i\cong H_{\mid Z_i}-t^*H_{\mid Z_i},
\end{eqnarray}
where $t$ is the translation of $A_c\cong Z_i$ given by the composition
$a_{\mid \widehat{Z}_{i,2}}\circ  \iota\circ (a_{\mid \widehat{Z}_{i,1}})^{-1}$, $\iota$ being  the natural  isomorphism $\widehat{Z}_{i,1}\cong \widehat{Z}_{i,2}$ given by the isomorphisms $$n_{\mid \widehat{Z}_{i,1 }}:\widehat{Z}_{i,1 }\cong Z_i ,\,\,n_{\mid \widehat{Z}_{i,2}}:   \widehat{Z}_{i,2}\cong Z_i.$$
\end{prop}
Note that, from (\ref{eqisopourMi14mai}), we conclude that  the line bundle $M_i$ is  torsion if and only if the translation $t$ has finite order. In most examples, the translation $t$ does not have finite order (see \cite{hwangoguiso} for examples). In fact, the translation $t$ is well understood from the viewpoint of degenerations of abelian varieties. If the family of fibers is the  Jacobian fibration of a family of  curves $\mathcal{C}\rightarrow B$ degenerating generically along the discriminant hypersurface $\Delta$ to irreducible curves $\mathcal{C}_{c}$, $c\in \Delta$   with one node $x$, then the translation used to construct the generalized Jacobian of  $\mathcal{C}_{c}$  starting from a $\mathbb{P}^1$-bundle on $J(\widehat{\mathcal{C}}_{c})$ by glueing two sections is translation by the point ${\rm alb}(x_1-x_2)\in J(\widehat{\mathcal{C}}_{c})$, where $x_1,\,x_2\in \widehat{\mathcal{C}}_{c}$ are the two preimages of $x$ under the normalization map $\widehat{\mathcal{C}}_{c}\rightarrow \mathcal{C}_{c}$  (see \cite{carlson}). Similarly, if  the family of fibers is the  Jacobian fibration of a family of threefolds  $\mathcal{X}\rightarrow B$ with $h^{3,0}=0$ degenerating to threefolds $\mathcal{X}_{c}$ with one ordinary double point $x_0$  and nonzero vanishing cycle, then  the similarly defined  translation is translation by the point $\Phi_{\widetilde{X}_c}(R_1-R_2)\in J^3(\widetilde{X}_c)$, where
$\widetilde{X}_c\rightarrow X_c$ is the blow-up of the singular point,  $\Phi_{\widetilde{X}_c} $ is its Abel-Jacobi map, and the curves $R_i$ are lines in the two different rulings of the exceptional divisor, which is a $2$-dimensional quadric (see \cite{collino}).

\begin{proof}[Proof of Proposition \ref{lepourMi14mai}]  Our assumption is that, locally analytically near $Z_i$, $X_c$ is the union of two smooth divisors in $X_C$ so that $n$ has degree $2$ over $Z_i$. For any point $x\in Z_i$, we denote by $x_1,\,x_2$ its preimages in $\widehat{X}_c$ and $R_{x_1},\,R_{x_2}$ the respective fibers of $a$ through $x_1,\,x_2$.  Let $g$ be a  coordinate on $C$ centered at $c$. Then $f^*g$ vanishes along $X_{c}$, and its differential vanishes along the component $Z_i$ of the singular locus of $X_c$. The Hessian of $f$ along $Z_i$ is a nonzero section of $S^2\mathcal{E}_{\mid Z_i}$ since our assumptions also imply that  $X_c$ has multiplicity $2$ along $Z_i$. If $M_i$ is not a torsion line bundle, the only nonzero section of $S^2\mathcal{E}$ is the section of the factor $M_i\otimes M_i^{-1}\subset S^2\mathcal{E}$, which vanishes along the two factors $M_i,\,M_i^{-1}$ of $\mathcal{E}_{\mid Z_i}^*$.  As we assumed  that $n$ is an immersion at the points $x_1,\,x_2$ above $x\in Z_i$ and we know that the inverse image $\widehat{Z}_i\subset \widehat{X}_c$ of $Z_i$  is \'{e}tale over $A_c$, we get up to a permutation  identifications $T_{R_{x_j},x_j}\cong M_{j}^{\epsilon_j}$ with $j=1,\,2$ and $\epsilon_j=(-1)^j$. This proves (\ref{eqiso14mai}).

As $X_c$ is irreducible, $\widehat{X}_c$ is connected and it is a $\mathbb{P}^1$-bundle over an abelian variety $A_c$ by \cite{hwangoguiso}. The inverse image $\widehat{Z}_i\subset \widehat{X}_c$ is \'{e}tale over $Z_i$, hence is a disjoint union of at most two abelian varieties, all isogenous to $A_c$. If $\widehat{Z}_i$ is irreducible, we have $n_*\mathcal{O}_{\widehat{Z}_i}=\mathcal{O}_{Z_i}\oplus \mathcal{O}_{Z_i}(\eta)$ for some $2$-torsion line bundle $\eta\in {\rm Pic}^0(Z_i) $.  By (\ref {eqiso14mai}), using the fact that the line bundle $\mathcal{O}_{\widehat{Z}_i}(\widehat{Z}_i)$   on $\widehat{Z}_i$ is homogeneous, hence topologically trivial, we get  in this case an isomorphism
$$\mathcal{E}_{\mid Z_i}\cong M'_i\oplus M'_i(\eta)$$
for some line bundle $M'_i\in {\rm Pic}^0(Z_i)$. Comparing with the isomorphism
$$  \mathcal{E}_{\mid Z_i}\cong M_i\oplus M_i^{-1}$$ of (i), we conclude that $M_i$ is  torsion, which is a contradiction. It follows that $\widehat{Z}_i$ has two connected components $\widehat{Z}_{i,1}, \,\widehat{Z}_{i,2}$, each dominating $A_c$ since otherwise they would be ruled. As the canonical bundle of  $X_c$ is trivial, the divisor $\widehat{Z}_i$ has degree at most $2$ over $A_c$ so  each component $\widehat{Z}_{i,1},\,\widehat{Z}_{i,2}$ has degree $1$ over $A_c$, hence is isomorphic to $A_c$.

It remains to prove (\ref{eqisopourMi14mai}). Let $H$ be an ample line bundle on $X_c$. Assume for simplicity that the pull-back $\widehat{H}:=n^*H$ has degree $1$ on the fibers of $a: \widehat{X}_c\rightarrow A_c$. (This can be assumed  in any case if we work with Picard groups with  $\mathbb{Q}$-coefficients.) Then we have equalities in ${\rm Pic}(\widehat{X}_c)$
\begin{eqnarray} \label{eqdu15mai} \mathcal{O}_{\widehat{X}_c}(\widehat{Z}_{i,1})=\widehat{H}\otimes a^*M'_1,\,\,  \mathcal{O}_{\widehat{X}_c}(\widehat{Z}_{i,2})=\widehat{H}\otimes a^*M'_2
\end{eqnarray}
for some line bundles $M'_i$ on $A_c$. We now use the fact that $\widehat{H}$ is the  pull-back of a line bundle under  $n$, so $\widehat{H}_{\mid  \widehat{Z}_{i,2}}$ and $\widehat{H}_{\mid  \widehat{Z}_{i,1}}$ coincide via the isomorphism $\iota$. Denoting $a_1:=a_{\mid  \widehat{Z}_{i,1}},\, a_2:=a_{\mid  \widehat{Z}_{i,2}}$, it follows from the  definition  of $t$ that
\begin{eqnarray} \label{eqdu15mai1} (a_2^{-1})^* \widehat{H}=t^*( (a_1^{-1})^*\widehat{H}).
\end{eqnarray}
We now observe that $\widehat{Z}_{i,1}$ and $\widehat{Z}_{i,2}$ do not intersect, so we get
$$\mathcal{O}_{\widehat{X}_c}(\widehat{Z}_{i,1})_{\mid \widehat{Z}_{i,2}}\cong \mathcal{O}_{\widehat{Z}_{i,2}},$$
hence, using (\ref{eqdu15mai}),  (\ref{eqdu15mai1})
\begin{eqnarray} \label{eqdu15mai1new} (a_2^{-1})^*\widehat{H}\otimes M'_1=\mathcal{O}_{A_c}=t^*((a_1^{-1})^*\widehat{H})\otimes M'_1.
\end{eqnarray}
It follows that
$$(a_1^{-1})^*(n^*M_i)=(a_1^{-1})^*(\mathcal{O}_{\widehat{Z}_{i,1}}(\widehat{Z}_{i,1}))=(a_1^{-1})^*\widehat{H}\otimes M'_1=(a_1^{-1})^*\widehat{H}\otimes t^*((a_1^{-1})^*\widehat{H})^{-1},$$
where the first equality follows from (\ref{eqiso14mai}) and the third from (\ref{eqdu15mai1new}).
This  proves (\ref{eqisopourMi14mai}) since $\widehat{H}=n^*H$.
\end{proof}
\begin{rema} \label{remaquisert} {\rm The proof also shows that the singular locus $Z$ is irreducible in this case. This will be used below.}
\end{rema}
We now give a criterion for the vanishing of the class $c_2(X)^2L^{n-1}$ in ${\rm CH}(X)$.
\begin{prop}\label{pronew14maicrit}  Let $X\rightarrow B$ be a Lagrangian fibered hyper-K\"ahler manifold. Assume
the following  three conditions hold:
\begin{enumerate}
\item\label{eqitem114mai} In codimension $1$ on $B$,  the singularities of the fibers are as in Proposition \ref{lepourMi14mai}, that is,   the  fiber $X_b$ is irreducible, it has at most two local branches at any point, and they are smooth.
\item \label{eqitem214mai} For any  irreducible component $\Delta_j$ of the discriminant hypersurface $\Delta\subset B$,  any desingularization $\widetilde{\Delta}_j$ of $\Delta_j$ satisfies  ${\rm CH}_0(\widetilde{\Delta}_j)=\mathbb{Z}$.
\item \label{eqitem314mai} For any  irreducible component $\Delta_j$ as above, the family of abelian varieties $A_b={\rm Alb}(\widehat{X}_b)$  parameterized by a Zariski open set of $\Delta_j$ satisfies the following nondegeneracy property:

(*)    For a general point  $b\in \Delta_j$, there exists $u\in T_{\Delta_j, b}$ such that the first order variation of Hodge structure
    $$\overline{\nabla}_u: H^{1,0}(A_b)\rightarrow H^{0,1}(A_b)$$
    is an isomorphism.
\end{enumerate}
Then  $c_2(X)^2L^{n-1}=0$ in ${\rm CH}(X)$.
\end{prop}
\begin{proof} The proof follows closely the argument of \cite{voisintorsion}. The family $\mathcal{A}_j =(A_{b})_{ b\in \Delta_j}$, of abelian varieties is an integrable system generically over $\Delta_j$ because the symplectic holomorphic form has generic rank $2n-2$ on the normalization of the restricted family $X_{\Delta_j}$, which is ruled over  $\mathcal{A}_j$, hence induces a  $2$-form on $\mathcal{A}_j$ which is everywhere nondegenerate over  a Zariski open set of $\Delta_j$. This $2$-form makes $\mathcal{A}_j\rightarrow \Delta_j$ a Lagrangian fibration over a Zariski open set of $\Delta_j$.  By symmetry (see \cite{DM})  of the map giving the infinitesimal variation of Hodge structure
$$\overline{\nabla}: T_{\Delta_j,b}\rightarrow {\rm Hom}\,(H^{1,0}(A_{j,b}), H^{0,1}(A_{j,b})),$$  the condition (*) translates into the fact that
for general $b\in \Delta_j$, and general $\lambda\in H^{1,0}(A_{j,b})$, the map
$$\overline{\nabla}(\lambda): T_{\Delta_{j},b}\rightarrow  H^{0,1}(A_{j,b})$$
is surjective.   By the main result of \cite{andrecorvajazannier}, this condition implies that,  for any normal function $\nu$, that is, algebraic section $\nu$ of ${\rm Pic}^0(\mathcal{A}'_j/\Delta'_j)\rightarrow \Delta_j'$ defined over a  generically  finite cover $\Delta'_j$ of $\Delta_j$, the set of points $b\in \Delta_j'$ such that $\nu(b)$ is  torsion in ${\rm Pic}^0(A_{j,b})$ is dense for the Euclidean topology of $\Delta_j$.
We apply this to the normal function defined as follows. Let $\mathcal{Z}_j$ be the singular locus of $X_{\Delta_j}$. If we restrict to an adequate dense Zariski open set of $\Delta_j$, $\mathcal{Z}_j$ is smooth of codimension $2$ in $X$ and its conormal bundle $\mathcal{E}_j$ has one of the forms (i), (ii) of Lemma \ref{lepouE}. In case (ii), we already noticed that  the contribution of the fibers ${Z}_{j,b}$ to $c_2(X)^2L^{n-1}$ is trivial. In case (i), the decomposition
$$\mathcal{E}_{j\mid \mathcal{Z}_{j,b}}=M_{j,b}\oplus M_{j,b}^{-1}$$
defines  a  degree $2$ cover of $\Delta_j$  parameterizing the choice of one of the line bundles $M_{j,b},\,M_{j,b}^{-1}\in {\rm Pic}^0(\mathcal{Z}_{j,b})={\rm Pic}^0(A_{j,b})$.
It follows that  there are points  $b\in \Delta_j$ where the line bundles $M_{j,b}$ and $M_{j,b}^{-1}$ are  torsion. By Lemmas \ref{proppourc22} and  \ref{lepouE}, the cycle
$c_2(\Omega_{X\mid Z_{j,b}})$ is trivial in ${\rm CH}^2(Z_{j,b})$ for any of these points. By condition \ref{eqitem214mai}, the cycle $j_{Z_{j,b}*}(c_2(\Omega_{X\mid Z_{j,b}}))\in {\rm CH}^{n+3}(X)$ does not depend on $b\in \Delta_j$, hence it is trivial for any $b\in \Delta_j$. The cycle
$c_2(\Omega_X)^2L^{n-1}$ is up to a coefficient the sum  of these  cycles over all $j$, and all points $b\in \Delta_j \cap C$, hence it is trivial.
\end{proof}
\begin{proof}[Proof of Theorem \ref{theoLSV}] We just have to check that the assumptions of Proposition \ref{pronew14maicrit} are satisfied in the case of the LSV manifold and its Lagrangian fibration. Let $X\subset \mathbb{P}^{5}$ be  a smooth cubic fourfold. The  base of the associated LSV manifold is the dual 5-dimensional projective space $\mathbb{P}^{5*}$ and  the discriminant hypersurface $\Delta\subset \mathbb{P}^{5*}$ is the dual of the cubic $X\subset   \mathbb{P}^5$. It follows that the LSV fibration  satisfies condition \ref{eqitem214mai}. The condition \ref{eqitem114mai} is also well-known : the  generic behavior  along $\Delta$ of the compactified intermediate Jacobian fibration is the same as the  generic behaviour, along the discriminant divisor, of the family of intermediate Jacobians of all cubic threefolds. This is  studied in \cite{collino}, \cite{DM}. It only remains to check condition \ref{eqitem314mai}. However, since we are considering a Lagrangian family of abelian fourfolds, it is observed in \cite{voisintorsion} that the infinitesimal condition \ref{eqitem314mai}, which says that for generic $b\in \Delta$, there exists a $u\in T_{\Delta,b}$ such that
$$\overline{\nabla}_u: H^{1,0}(A_b)\rightarrow H^{0,1}(A_b)$$ is an isomorphism,
is equivalent to the fact that there does not exist any $u\in T_{\Delta,b},\,u\not=0$, such that
$\overline{\nabla}_u: H^{1,0}(A_b)\rightarrow H^{0,1}(A_b)$ is identically $0$, that is, the moduli map of the family of abelian varieties $(A_{b})_{b\in \Delta}$ has nowhere  maximal rank. This follows from a result due to Lossen  \cite{lossen}  concerning  the projective geometry of homogenenous cubic polynomials  in four variables: it says that if a projective cubic surface is not a cone, then it has a point where its Hessian (or second fundamental form) is a nondegenerate quadratic form. We apply the Lossen result to the cubic form
$$C(u,v,w):=\langle \overline{\nabla}_u(v),w\rangle,$$
with $v\in H^{1,0}(A_b),\,u\in T_{\Delta,b},\,w\in H^{1,0}(A_b)$. The symmetry of $C$ in $u,\,v,\,w$, using the natural isomorphism $T_{\Delta,b}\cong H^{1,0}(A_b)$ given by the Lagrangian fibration structure, is observed in \cite{DM}.

By this argument, if  condition \ref{eqitem314mai} were not satisfied, the family of abelian varieties $(A_{b})_{b\in \Delta}$ would not have maximal modulus, and it is easy to show that this does not happen, at least for very general  $X$. For example, one can  note that, by \cite{CG}, these abelian fourfolds are the Jacobians of complete intersection curves of type $(2,3)$ in $\mathbb{P}^3$, so that their variation of Hodge structure can be described  explicitly.
\end{proof}

\begin{rema} {\rm If instead of a LSV manifold we consider the punctual Hilbert scheme $S^{[g]}$ of a $K3$ surface $S$ with Picard group generated by a line bundle of self-intersection $2g-2$, or rather its birational version admitting a Lagrangian fibration compactifying the Jacobian fibration of the universal family of curves over $|L|$ (see \cite{beauville}), then the proof above does not apply since the condition \ref{eqitem214mai} in Proposition \ref{pronew14maicrit} is not satisfied. However, in this case the discriminant hypersurface is birational to a projective bundle over $S$, hence the relations in the Chow group of a $K3$ surface established  in \cite{beauvoi} (see statements \ref{i} and \ref{ii}  in the introduction) can probably be used to conclude also  in that case. As the result is proved in \cite{mauliknegut}, we do not pursue this argument.}
\end{rema}
\subsection{Riess' argument and  proof of Theorem \ref{theoriess}\label{secriess}}
Theorem \ref{theoriess} follows from Theorem \ref{theonew} by Riess'  arguments as in \cite{riess}.
For completeness, we sketch the proof below.
The main result from Riess' paper that we need is the following.
\begin{theo} \label{theorefriess}  Let $X$ be a projective hyper-K\"{a}hler manifold of dimension $2n$, and let $L\in{\rm NS}(X)$ be an isotropic class. Then there exists a projective hyper-K\"{a}hler manifold $X'$, and a correspondence $\Gamma\in{\rm CH}^{2n}(X\times X')$ such that
\begin{enumerate}
\item \label{iriess} $\Gamma$ induces a graded  ring isomorphism $\Gamma_*: {\rm CH}(X)\rightarrow  {\rm CH}(X')$ which maps $c_i(X)$ to $c_i(X')$.
\item  \label{iiriess} $\Gamma_*(L)=L'$, where  for some nef  line bundle $L'$ on $X'$.
\end{enumerate}
\end{theo}
 The variety $X'$ is birational to $X$, hence deformation equivalent to $X$ by \cite{huy}. The cycle $\Gamma$ is effective and is the limit of graphs of isomorphisms $X_t\cong X'_t$ for some small  deformations $X_t$ of $X$, resp. $X'_t$ of $X'$. Such isomorphisms satisfy property \ref{iriess}, hence also $\Gamma$, as was observed by Riess in \cite{riess2}. Being the specialization of graphs of isomorphisms, the correspondence $[\Gamma]^*$ also induces a ring isomorphism on cohomology, so $L'$ is also isotropic since this is equivalent to ${\rm deg}\,{L'}^{2n}=0$ by the Beauville-Fujiki formula, and we have ${\rm deg}\,{L}^{2n}=0$.

Assume now the SYZ conjecture for the hyper-K\"{a}hler manifolds of the same deformation type as $X$. Let $l=c_1(L)$ be an algebraic  isotropic class on  $X$. By Theorem \ref{theorefriess}, there exist $X',\,\Gamma\in{\rm CH}^{2n}(X\times X')$ satisfying properties \ref{iriess} and \ref{iiriess}. As $l':=\Gamma_*l$ is nef on $X'$, and $X'$ satisfies the SYZ conjecture, one has by Theorem \ref{theonew}
$${L'}^{n+1-i}c_{2j}(X')=0\,\,{\rm in}\,\,{\rm CH}(X')\,\,\,{\rm for}\,\,\, j\geq i.$$

As $\Gamma$ is a ring isomorphism preserving Chern classes, it follows that
$${L}^{n+1-i}c_{2j}(X)=0\,\,{\rm in}\,\,{\rm CH}(X),$$
for $j\geq i$,
which proves Theorem \ref{theoriess}.


\begin{thebibliography}{99}
\bibitem{andrecorvajazannier}  Y. Andr\'e, P. Corvaja, U. Zannier. The Betti map associated to a section of an abelian scheme
(with an appendix by Z. Gao),  Invent. Math. 222 (2020), no. 1, 161-202.
\bibitem{arinkinfedorov}  D. Arinkin,  R. Fedorov. Partial Fourier-Mukai transform for integrable systems with applications to
Hitchin fibration, Duke Math. J. 165 (2016), no. 15, 2991-3042.
\bibitem{beauspli} A. Beauville.  On the splitting of the Bloch-Beilinson filtration, in {\it Algebraic cycles and motives}. Vol. 2, 38-53, London Math. Soc. Lecture Note Ser., 344, Cambridge Univ. Press, Cambridge, 2007.
\bibitem{beauvoi} A. Beauville, C. Voisin. On the Chow ring of a K3 surface,  J. Algebraic Geom. 13 (2004), no. 3, 417-426.
    \bibitem{beauville} A. Beauville. Counting rational curves on K3 surfaces. Duke Math. J. 97, 99-108 (1999).
\bibitem{bloch} S. Bloch. {\it Lectures on algebraic cycles}, Duke University Mathematics Series, IV. Duke University, Mathematics Department, Durham, N.C., 1980.
    \bibitem{blochsrinivas} S. Bloch, V. Srinivas. Remarks on
 correspondences and algebraic cycles,
 Amer. J. of Math. 105 (1983) 1235-1253.
 \bibitem{bogomolov}  F. A. Bogomolov. On the cohomology ring of a simple hyper-K\"ahler manifold (on the results of Verbitsky).
Geom. Funct. Anal.  6 (4), 612-618, 1996.
\bibitem{brion} M. Brion. Some structure theorems for algebraic groups, in {\it Algebraic groups: structure and actions}, 53–126, Proc. Sympos. Pure Math., 94, Amer. Math. Soc., Providence, RI, (2017).
    \bibitem{carlson} J. Carlson. Extension of mixed Hodge structure, in {\it Journ\'{e}es de g\'{e}om\'{e}trie alg\'{e}brique d'Angers}, Ed. A. Beauville Rockville, Sijthoff and  Noordhoff (1980).
    \bibitem{claudon}  B. Claudon, F.  Loray,  J. Pereira, F. Touzet.  Compact leaves of codimension one holomorphic foliations on projective manifolds. Ann. Sci. \'{E}c. Norm. Sup\'{e}r. (4) 51 (2018), no. 6, 1457-1506.
        \bibitem{CG} H. Clemens, Ph. Griffiths. The intermediate Jacobian of the cubic threefold. Ann. of Math. (2) 95 (1972), 281-356.
        \bibitem{collino}  A. Collino.
A cheap proof of the irrationality of most cubic threefolds,
Boll. Un. Mat. Ital. B (5) 16 (1979), no. 2, 451-465.
 \bibitem{DHMV} O. Debarre, D. Huybrechts, E. Macr\`{\i}, C. Voisin.  Computing Riemann-Roch polynomials and classifying  hyper-K\"{a}hler fourfolds,  arXiv:2201.08152, to appear in JAMS.
     \bibitem{DM}  R. Donagi, E. Markman. Spectral covers, algebraically completely integrable, Hamiltonian systems, and moduli of bundles, in {\it  Integrable systems and quantum groups }(Montecatini Terme, 1993), 1-119, Lecture Notes in Math., 1620, Springer, Berlin, 1996.
     \bibitem{druel} S. Druel.
Codimension 1 foliations with numerically trivial canonical class on singular spaces,
Duke Math. J. 170 (2021), no. 1, 95-203.
\bibitem{liefu} L. Fu.
Beauville-Voisin conjecture for generalized Kummer varieties,
Int. Math. Res. Not. IMRN 2015, no. 12, 3878-3898.
     \bibitem{fulton} W. Fulton.  {\it  Intersection theory}. Ergebnisse der Mathematik und ihrer Grenzgebiete (3) [Results in Mathematics and Related Areas (3)], 2. Springer-Verlag, Berlin, (1984).

 \bibitem{huy} D. Huybrechts. Compact hyper-K\"{a}hler manifolds: basic results. Invent. Math. 135 (1999), no. 1, 63-113.
\bibitem{huybrechts} D. Huybrechts.  The K\"ahler cone of a compact hyperk\"{a}hler manifold. Math. Ann. 326 (2003), no. 3, 499-513.
    \bibitem{huyxu}   D. Huybrechts, Ch. Xu.  Lagrangian fibrations of hyperk\"{a}hler fourfolds. J. Inst. Math. Jussieu 21 (2022), no. 3, 921-932.
\bibitem{hwang}  J.-M  Hwang. Base manifolds for fibrations of projective irreducible symplectic manifolds, Invent.
Math. 174 (2008), 625-644.
\bibitem{hwangoguiso} J.-M.  Hwang, K. Oguiso.
Characteristic foliation on the discriminant hypersurface of a holomorphic Lagrangian fibration,
Amer. J. Math. 131 (2009), no. 4, 981-1007.
\bibitem{hwangoguisomultiple} J.-M.  Hwang, K. Oguiso.
Multiple fibers of holomorphic Lagrangian fibrations.
Commun. Contemp. Math. 13 (2011), no. 2, 309-329.
 \bibitem{jouanolou} J.-P. Jouanolou. Une suite exacte de Mayer-Vietoris en K-th\'{e}orie alg\'{e}brique. In H. Bass, editor, {\it Higher K-Theories},
 Lecture Notes in Math.  volume 341, pages 293-316,  Springer-Verlag, Berlin Heidelberg New York Tokyo, (1973).
\bibitem{kim}  Y. Kim. The dual Lagrangian fibration of known hyper-K\"{a}hler manifolds, arXiv:2109.03987.
\bibitem{kim2} Y. Kim. In preparation.
\bibitem{lsv} R. Laza, G. Sacc\`a,  C. Voisin. A hyper-K\"ahler compactification of the intermediate jacobian fibration
associated with a cubic 4-fold. Acta Math.  218 (1), 55-135, 2017.
\bibitem{lossen}  Ch. Lossen. When does the Hessian determinant vanish identically ? Bull. Braz. math. Soc.
35, 71-32 (2004).
\bibitem{matsushita} D. Matsushita. On deformations of Lagrangian fibrations, in  {\it K3 surfaces and their moduli}, 237-243, Progr. Math., 315, Birkh\"{a}user/Springer, [Cham], (2016).
\bibitem{matsuflat}  D. Matsushita. Higher direct images of dualizing sheaves of Lagrangian fibrations.
Amer. J. Math. 127 (2005), no. 2, 243-259.
\bibitem{mauliknegut} D. Maulik, A. Negut. Lehn's formula in Chow and conjectures of Beauville and Voisin. J. Inst. Math. Jussieu 21 (2022), no. 3, 933-971.
\bibitem{mukai}  S. Mukai. Semi-homogeneous vector bundles on an abelian variety, J. Math. Kyoto
Univ. 18 (1978), no. 2, 239-272.
\bibitem{mumford} D. Mumford.  Rational equivalence of 0-cycles on surfaces. J. Math. Kyoto Univ. 9
(1968) 195-204.
\bibitem{niep} M. A. Nieper-Wisskirchen. On the Chern numbers of generalised Kummer varieties. Math. Res. Lett. 9 (2002), no. 5-6, 597-606.
\bibitem{osv}  G. Oberdieck, J.  Song, C. Voisin. Hilbert schemes of K3 surfaces, generalized Kummer, and
cobordism classes of hyper-K\"ahler manifolds,   Pure Appl. Math. Q. 18 (2022), no. 4, 1723-1748.

 \bibitem{riess} U. Riess. On Beauville's conjectural weak splitting property. Int. Math. Res. Not. IMRN 2016, no. 20, 6133-6150.
 \bibitem{riess2} U. Riess.  On the Chow ring of birational irreducible symplectic varieties. Manuscripta Math. 145 (2014), no. 3-4, 473-501.
     \bibitem{roitman} A. A. Roitman. Rational equivalence of zero-dimensional cycles. (Russian) Mat. Zametki 28 (1980), no. 1, 85-90, 169.
\bibitem{sacca} G. Sacc\`a. Birational geometry of the intermediate Jacobian fibration of a cubic fourfold, Geom. Topol. 27 (2023), no. 4, 1479-1538.
    \bibitem{voisinpamq} C. Voisin. On the Chow ring of certain algebraic hyper-K\"{a}hler manifolds, Pure and Applied Mathematics Quarterly, Volume 4, Number 3, (Special issue in honor of Fedya Bogomolov), (2008).
\bibitem{voicitrouille} C. Voisin. {\it Chow rings, decomposition of the diagonal  and the topology of families},
Annals of Math. Studies 187,  Princeton University Press 2014.
\bibitem{voisintorsion} C. Voisin. Torsion points of sections of Lagrangian torus fibrations and the Chow ring of hyper-K\"{a}hler manifolds, in {\it  Geometry of moduli}, 295-326, Abel Symp., 14, Springer, Cham, 2018.

\end{thebibliography}
    \end{document}